\documentclass[reqno,11pt]{amsart}

\usepackage{a4wide}
\usepackage[english]{babel}
\usepackage[utf8]{inputenc}
\usepackage{amsmath,amssymb}
\usepackage{tikz}
\usepackage[hidelinks]{hyperref}
\usepackage{dsfont}

\theoremstyle{definition}
\newtheorem{thm}{Theorem}[section]
\newtheorem{lem}[thm]{Lemma}
\newtheorem{prop}[thm]{Proposition}
\newtheorem{defi}[thm]{Definition}
\newtheorem{cor}[thm]{Corollary}
\newtheorem{rem}[thm]{Remark}
\numberwithin{equation}{section}

\newcommand{\Sc}{\operatorname{Sc}}
\renewcommand{\Im}{\operatorname{Im}}
\newcommand{\ran}{\operatorname{ran}}
\newcommand{\dom}{\operatorname{dom}}
\newcommand{\graph}{\operatorname{graph}}
\newcommand{\Arg}{\operatorname{Arg}}
\newcommand{\poly}{{\text{poly}}}
\newcommand\essran{\operatorname{essran}}
\newcommand\esssup{\operatorname{esssup}}

\title[]{The $S$-functional calculus for the \\ Clifford adjoint operator}

\author[Fabrizio Colombo]{Fabrizio Colombo}
\address{(FC) Politecnico di Milano, Dipartimento di Matematica, Via E. Bonardi 9, 20133 Milano, Italy}
\email{fabrizio.colombo@polimi.it}

\author[Francesco Mantovani]{Francesco Mantovani}
\address{(FM) Politecnico di Milano, Dipartimento di Matematica, Via E. Bonardi 9, 20133 Milano, Italy}
\email{francesco.mantovani@polimi.it}

\author[Peter Schlosser]{Peter Schlosser}
\address{(PS) Graz University of Technology, Institute of Applied Mathematics, Steyrergasse 30, 8010 Graz, Austria}
\email{pschlosser@math.tugraz.at}

\begin{document}

\begin{abstract}
In this paper, we work within the framework of modules over the Clifford algebra $\mathbb{R}_d$. Our investigation focuses on the $S$-spectrum and the $S$-functional calculus in its various forms for the adjoint $T^*$ of a Clifford operator $T$. One of the key results we present is that the bisectoriality of $T$ can be transferred to $T^*$. This is grounded in the fact that, for Clifford operators, the $S$-spectrum of the adjoint operator $T^*$ is identical to that of $T$. Furthermore, we demonstrate that for the existing formulations of the $S$-functional calculus, including bounded, unbounded, $\omega$, and $H^\infty$ versions, there is a clear connection between the left functional calculus of $T$ and the right functional calculus of $T^*$. This explicit link between the left functional calculus of $T$ and the right functional calculus of $T^*$ and vice versa is obtained using the function $f^\#(s):=\overline{f(\overline{s})}$. Finally, we discuss the fact that the $S$-spectrum, related to the invertibility of the second-order operator $T^2-2s_0T+|s|^2$, can actually be characterized through the invertibility of the first-order $\mathbb{R}$-linear operator $T-\mathcal{I}^Rs$.
\end{abstract}

\maketitle

AMS Classification 47A10, 47A60. \medskip

Keywords: $S$-spectrum, Adjoint operator, $S$-functional calculus, $H^\infty$-functional calculus, Multiplication operator.

\section{Introduction}

Quaternionic and Clifford operators are essential in several areas of mathematics and physics, including quaternionic quantum mechanics, vector analysis, differential geometry, and hypercomplex analysis, for more details and references see Section~\ref{SEC_CONCLUD_RMK}. The spectral theory for Clifford operators differs from the classical complex theory. This distinction arises from the noncommutativity of the spectral parameter $s\in\mathbb{R}^{d+1}$ and the operator $T$, and manifests itself in the definition of the spectrum, which is related to the bounded invertibility of the second order operator
\begin{equation}\label{Eq_Qs}
Q_s[T]:=T^2-2s_0T+|s|^2,\qquad\text{with }\dom(Q_s[T]):=\dom(T^2).
\end{equation}
As discussed in \cite{ColomboSabadiniStruppa2011,ColSab2006} and detailed in \cite{FJBOOK,CGK}, the $S$-spectrum of $T$ is then defined as
\begin{equation}\label{Eq_S_spectrum_intro}
\sigma_S(T):=\big\{s\in\mathbb{R}^{d+1}\;\big|\;Q_s[T]^{-1}\notin\mathcal{B}(V)\big\},
\end{equation}
where $\mathcal{B}(V)$ denotes the set of bounded, everywhere defined operators on the Clifford module $V$. In this context, for every $s\notin\sigma_S(T)$, the two $S$-resolvent operators, both left and right, are defined as
\begin{equation}\label{Eq_S_resolvent_operators}
S_L^{-1}(s,T):=Q_s[T]^{-1}\overline{s}-TQ_s[T]^{-1}\qquad\text{and}\qquad S_R^{-1}(s,T):=(\overline{s}-T)Q_s[T]^{-1}.
\end{equation}
These resolvent operators are now crucial for defining the $S$-functional calculus, which is the Clifford analog of the Riesz-Dunford functional calculus in \cite{RD}. Roughly speaking, the $S$-functional calculus is defined as
\begin{subequations}\label{Eq_SFC}
\begin{align}
f(T):=&\frac{1}{2\pi}\int_{\partial U\cap\mathbb{C}_J}S_L^{-1}(s,T)ds_Jf(s),\qquad\text{or} \label{Eq_SFC_left} \\
f(T):=&\frac{1}{2\pi}\int_{\partial U\cap\mathbb{C}_J}f(s)ds_JS_R^{-1}(s,T), \label{Eq_SFC_right}
\end{align}
\end{subequations}
where $U$ is a suitable open set containing the $S$-spectrum of $T$, and $J$ is some arbitrary imaginary unit of the Clifford algebra. Since in a Clifford algebra one distinguishes between left and right slice hyperholomorphic functions, we need two different versions of the $S$-functional calculi, \eqref{Eq_SFC_left} for left slice hyperholomorphic functions $f$, and \eqref{Eq_SFC_right} for right slice hyperholomorphic functions. \medskip

Depending on whether the operator $T$ is bounded or unbounded, and depending on the decay properties of the function $f$, the integrals \eqref{Eq_SFC} have to be interpreted differently. Precise definitions of the several cases are given in the subsections of Section~\ref{sec_SFC}. In particular, since the $S$-spectrum for unbounded operators is in general not bounded, one has to adapt the integration path accordingly, or impose decay conditions on the function $f$, in order to ensure the convergence of the integrals \eqref{Eq_SFC}, see Definition~\ref{defi_Unbounded_SFC} or Definition~\ref{defi_omega_SFC}. \medskip

Only a regularization procedure, called $H^\infty$-functional calculus, then allows to extend \eqref{Eq_SFC} to functions for which the integrals diverge. In the complex setting, the $H^\infty$-functional calculus was first introduced by A. McIntosh in \cite{McI1}, with further developments and applications discussed in \cite{MC10,MC97,MC06,MC98}. It is a key instrument in spectral theory, particularly for differential operators, and it has applications in various fields as parabolic evolution equations, Kato’s square root problem, elliptic systems, and Schrödinger operators, see \cite{Haase,HYTONBOOK1,HYTONBOOK2}. \medskip

In the Clifford setting, the $H^\infty$-functional calculus for left slice hyperholomorphic functions is already well established in \cite{ACQS2016,CGdiffusion2018,FJBOOK}. It is given by
\begin{equation}\label{Eq_Hinfty_left}
f(T):=e(T)^{-1}(ef)(T),
\end{equation}
where both, the regularizer $e(T)$ and the regularized operator $(ef)(T)$ are interpreted as the integral \eqref{Eq_SFC_left}. Detailed properties of this calculus are examined in \cite{MS24}, based on the universality property of the $S$-functional calculus established in \cite{ADVCGKS}. The $H^\infty$-functional calculus defines in general unbounded operators, while in \cite{Quadratic}, boundedness has been characterized using suitable quadratic estimates. \medskip

However, it was unknown until the recent paper \cite{RIGHT}, how the regularization formula \eqref{Eq_Hinfty_left} adapts for right slice holomorphic functions. The key idea is to look at the right $H^\infty$-functional calculus as
\begin{equation}\label{Eq_Hinfty_right}
f(T):=\overline{(fe)(T)e(T)^{-1}},
\end{equation}
where it is a priori not known if $(fe)(T)e(T)^{-1}$ is closable. Hence the closure leads to a multivalued operator $f(T)$ in the sense of Definition~\ref{defi_Multivalued_operators}. \medskip

The primary aim of the present article is to investigate the $S$-spectral theory of the adjoint operator $T^*$ of a densely defined right-linear operator $T:\dom(T)\subseteq V\rightarrow V$, acting in a Hilbert module $V$ over the Clifford algebra $\mathbb{R}_d$. In particular, we prove in Theorem~\ref{thm_Spectrum_Tstar} that the $S$-spectrum \eqref{Eq_S_spectrum_intro} of the adjoint operator coincides with the one of the original operator
\begin{equation*}
\sigma_S(T^*)=\sigma_S(T),
\end{equation*}
and also that the respective $S$-resolvent operators of $T^*$ and $T$ are connected via
\begin{equation*}
S_L^{-1}(s,T^*)=S_R^{-1}(\overline{s},T)^*\qquad\text{and}\qquad S_R^{-1}(s,T^*)=S_L^{-1}(\overline{s},T)^*.
\end{equation*}
We will also investigate the various versions of the $S$-functional calculus \eqref{Eq_SFC} of $T^*$. Starting with bounded operators in Theorem~\ref{thm_Bounded_Tstar}, continuing with unbounded operators in Theorem~\ref{thm_Unbounded_Tstar} and heading to the $\omega$-functional calculus of unbounded bisectorial operators in Theorem~\ref{thm_omega_SFC_Tstar}. In all the cases, we will show the connection
\begin{equation}\label{Eq_fTstar}
f(T^*)=f^\#(T)^*,
\end{equation}
between the functional calculus of $T^*$ and $T$, using the function
\begin{equation*}
f^\#(s):=\overline{f(\overline{s})}.
\end{equation*}
We also prove in Theorem~\ref{thm_Hinfty_Tstar}, that the formula \eqref{Eq_fTstar} holds for the $H^\infty$-functional calculus \eqref{Eq_Hinfty_left} and \eqref{Eq_Hinfty_right}, if one interprets both sides accordingly as unbounded or multivalued operators. \medskip

Additionally, we present some preliminary results in Section~\ref{sec_Preliminaries} in a more general setting for Banach modules, which are of independent interest. The most important one of them is Theorem~\ref{thm_rhoS_via_Is}, which characterizes the $S$-spectrum in terms of a first order operator. More precisely, for every $s\in\mathbb{R}^{d+1}$, with the right multiplication operator
\begin{equation}\label{Eq_IRs}
(\mathcal{I}^Rs)(v):=vs,\qquad v\in V,
\end{equation}
we can write the $S$-spectrum as
\begin{equation*}
\sigma_S(T)=\big\{s\in\mathbb{R}^{d+1}\;\big|\;(T-\mathcal{I}^Rs)^{-1}\notin\mathcal{B}(V)\big\}.
\end{equation*}
Note that since $T$ is right-linear and $\mathcal{I}^Rs$ is left-linear, their difference turns out to be $\mathbb{R}$-linear only. The idea of such a first order representation of the $S$-spectrum was already discussed in \cite[Theorem 3.1.8]{FJBOOK}, where a similar result was proven, but only on slices $\sigma_S(T)\cap\mathbb{C}_J$. \medskip

In the final Section~\ref{sec_Multiplication_operator} of this article, we present the multiplication operator $M_h$ with a Clifford valued function $h:X\rightarrow\mathbb{R}^{d+1}$ as an application of the $S$-functional calculus. In Theorem~\ref{thm_Spectrum_of_Mh}, we compute its $S$-spectrum, which turns out to not be the essential range as in the complex case, but rather the rotation of the essential range around the real axis
\begin{equation*}
\sigma_S(M_h)=[\essran(h)].
\end{equation*}
This fact is a consequence of the axial symmetry of the $S$-spectrum. Moreover, in Theorem~\ref{thm_Hinfty_Mh} we explicitly calculate the action
\begin{equation*}
f(M_h)=M_{f\circ h}
\end{equation*}
of the $S$-functional calculus on the multiplication operator. This formula means that the functional calculus $f(M_h)$ is again a multiplication operator with the new function $f\circ h$.

\section{Preliminaries on Clifford Algebras and Clifford modules}\label{sec_Preliminaries}

In this section we will fix the algebraic and functional analytic setting of this paper. The underlying algebra will be the real \textit{Clifford algebra} $\mathbb{R}_d$ over $d$ \textit{imaginary units} $e_1,\dots,e_d$, which satisfy the relations
\begin{equation*}
e_i^2=-1\qquad\text{and}\qquad e_ie_j=-e_je_i,\qquad i\neq j\in\{1,\dots,d\}.
\end{equation*}
More precisely, $\mathbb{R}_d$ is given by
\begin{equation*}
\mathbb{R}_d:=\Big\{\sum\nolimits_{A\in\mathcal{A}}s_Ae_A\;\Big|\;s_A\in\mathbb{R},\,A\in\mathcal{A}\Big\},
\end{equation*}
using the index set
\begin{equation*}
\mathcal{A}:=\big\{(i_1,\dots,i_r)\;\big|\;r\in\{0,\dots,d\},\,1\leq i_1<\dots<i_r\leq d\big\},
\end{equation*}
and the \textit{basis vectors} $e_A:=e_{i_1}\dots e_{i_r}$. Note that for $A=\emptyset$ the empty product of imaginary units is the real number $e_\emptyset=1$. Furthermore, we define for every Clifford number $s\in\mathbb{R}_d$ its \textit{conjugate} and its \textit{absolute value}
\begin{equation}\label{Eq_Conjugate_Norm}
\overline{s}:=\sum\nolimits_{A\in\mathcal{A}}(-1)^{\frac{|A|(|A|+1)}{2}}s_Ae_A\qquad\text{and}\qquad|s|^2:=\sum\nolimits_{A\in\mathcal{A}}s_A^2.
\end{equation}
An important subset of the Clifford numbers are the so called \textit{paravectors}
\begin{equation*}
\mathbb{R}^{d+1}:=\big\{s_0+s_1e_1+\dots+s_de_d\;\big|\;s_0,s_1,\dots,s_d\in\mathbb{R}\big\}.
\end{equation*}
For any paravector $s\in\mathbb{R}^{d+1}$, we define the \textit{imaginary part}
\begin{equation*}
\Im(s):=s_1e_1+\dots+s_de_d.
\end{equation*}
Also for every $s\in\mathbb{R}^{d+1}$, the conjugate and the modulus in \eqref{Eq_Conjugate_Norm} reduce to
\begin{equation*}
\overline{s}=s_0-s_1e_1-\dots-s_de_d\qquad\text{and}\qquad|s|^2=s_0^2+s_1^2+\dots+s_d^2.
\end{equation*}
The sphere of \textit{imaginary units} is defined by
\begin{equation*}
\mathbb{S}:=\big\{J\in\mathbb{R}^{d+1}\;\big|\;J_0=0,\,|J|=1\big\}.
\end{equation*}
Any element $J\in\mathbb{S}$ satisfies $J^2=-1$ and hence the corresponding \textit{hyperplane}
\begin{equation*}
\mathbb{C}_J:=\big\{x+Jy\;\big|\;x,y\in\mathbb{R}\big\}
\end{equation*}
is an isomorphic copy of the complex numbers. Moreover, for every paravector $s\in\mathbb{R}^{d+1}$ we consider the corresponding \textit{$(d-1)$--sphere}
\begin{equation*}
[s]:=\big\{s_0+J|\Im(s)|\;\big|\;J\in\mathbb{S}\big\}.
\end{equation*}
A subset $U\subseteq\mathbb{R}^{d+1}$ is called \textit{axially symmetric}, if $[s]\subseteq U$ for every $s\in U$. \medskip

Next, we introduce the notion of slice hyperholomorphic functions $f:U\rightarrow\mathbb{R}_d$, defined on an axially symmetric open set $U\subseteq\mathbb{R}^{d+1}$.

\begin{defi}[Slice hyperholomorphic functions]
Let $U\subseteq\mathbb{R}^{d+1}$ be open, axially symmetric and consider
\begin{equation*}
\mathcal{U}:=\big\{(x,y)\in\mathbb{R}^2\;\big|\;x+\mathbb{S}y\subseteq U\big\}.
\end{equation*}
A function $f:U\rightarrow\mathbb{R}_d$ is called \textit{left} (resp. \textit{right}) \textit{slice hyperholomorphic}, if there exist continuously differentiable functions $f_0,f_1:\mathcal{U}\rightarrow\mathbb{R}_d$, such that for every $(x,y)\in\mathcal{U}$:

\begin{enumerate}
\item[i)] The function $f$ admits for every $J\in\mathbb{S}$ the representation
\begin{equation}\label{Eq_Holomorphic_decomposition}
f(x+Jy)=f_0(x,y)+Jf_1(x,y),\quad\Big(\text{resp.}\;f(x+Jy)=f_0(x,y)+f_1(x,y)J\Big).
\end{equation}
\item[ii)] The functions $f_0,f_1$ satisfy the \textit{compatibility conditions}
\begin{equation}\label{Eq_Symmetry_condition}
f_0(x,-y)=f_0(x,y)\qquad\text{and}\qquad f_1(x,-y)=-f_1(x,y).
\end{equation}
\item[iii)] The functions $f_0,f_1$ satisfy the \textit{Cauchy-Riemann equations}
\begin{equation}\label{Eq_Cauchy_Riemann_equations}
\frac{\partial}{\partial x}f_0(x,y)=\frac{\partial}{\partial y}f_1(x,y)\qquad\text{and}\qquad\frac{\partial}{\partial y}f_0(x,y)=-\frac{\partial}{\partial x}f_1(x,y).
\end{equation}
\end{enumerate}

The class of left (resp. right) slice hyperholomorphic functions on $U$ is denoted by $\mathcal{SH}_L(U)$ (resp. $\mathcal{SH}_R(U)$). In the special case that $f_0$ and $f_1$ are real valued, we call $f$ \textit{intrinsic} and the space of intrinsic functions by $\mathcal{N}(U)$.
\end{defi}

Next, we turn our attention to modules over $\mathbb{R}_d$. For a real Banach space $V_\mathbb{R}$ with norm $\Vert\cdot\Vert_\mathbb{R}$, we define the corresponding \textit{Banach module}
\begin{equation*}
V:=\Big\{\sum\nolimits_{A\in\mathcal{A}}v_A\otimes e_A\;\Big|\;v_A\in V_\mathbb{R},\,A\in\mathcal{A}\Big\},
\end{equation*}
and equip it with the \textit{norm}
\begin{equation}\label{Eq_Norm}
\Vert v\Vert^2:=\sum\nolimits_{A\in\mathcal{A}}\Vert v_A\Vert_\mathbb{R}^2,\qquad v\in V.
\end{equation}
For any vector $v=\sum_{A\in\mathcal{A}}v_A\otimes e_A\in V$ and any Clifford number $s=\sum_{B\in\mathcal{A}}s_Be_B\in\mathbb{R}_d$, we establish the left and the right scalar multiplication
\begin{align*}
sv:=&\sum\nolimits_{A,B\in\mathcal{A}}(s_Bv_A)\otimes(e_Be_A), && \textit{(left-multiplication)} \\
vs:=&\sum\nolimits_{A,B\in\mathcal{A}}(v_As_B)\otimes(e_Ae_B). && \textit{(right-multiplication)}
\end{align*}
We recall the well known properties of these products, see for example \cite[Lemma 2.1]{MS24}:
\begin{align*}
\Vert sv\Vert&\leq 2^{\frac{d}{2}}|s|\Vert v\Vert\qquad\text{and}\qquad\Vert vs\Vert\leq 2^{\frac{d}{2}}|s|\Vert v\Vert,\qquad\text{if }s\in\mathbb{R}_d, \\
\Vert sv\Vert&=\Vert vs\Vert=|s|\Vert v\Vert,\hspace{4.83cm}\text{if }s\in\mathbb{R}^{d+1}.
\end{align*}
If $V_\mathbb{R}$ is also a real Hilbert space, i.e. equipped with an inner product $\langle\cdot,\cdot\rangle_\mathbb{R}$, we also make $V$ a \textit{Hilbert module} with the inner product
\begin{equation}\label{Eq_Inner_product}
\langle v,w\rangle:=\sum\nolimits_{A,B\in\mathcal{A}}\langle v_A,w_B\rangle_\mathbb{R}\overline{e_A}e_B,\qquad v,w\in V.
\end{equation}
The sesquilinear form \eqref{Eq_Inner_product} is now right-linear in the second, and right-antilinear in the first argument, i.e. for every $u,v,w\in V$, $s\in\mathbb{R}_d$, there holds
\begin{equation}\label{Eq_Inner_product_linearity}
\begin{split}
\langle u,v+w\rangle&=\langle u,v\rangle+\langle u,w\rangle,\hspace{1.5cm} \langle v,ws\rangle=\langle v,w\rangle s, \\
\langle v+w,u\rangle&=\langle v,u\rangle+\langle w,u\rangle,\hspace{1.5cm} \langle vs,w\rangle=\overline{s}\langle v,w\rangle.
\end{split}
\end{equation}
Moreover, there also holds
\begin{equation}\label{Eq_Inner_product_property}
\overline{\langle v,w\rangle}=\langle w,v\rangle\qquad\text{and}\qquad\langle v,sw\rangle=\langle\overline{s}v,w\rangle.
\end{equation}
The norm \eqref{Eq_Norm} is then connected to the inner product \eqref{Eq_Inner_product} by
\begin{equation*}
\Vert v\Vert^2=\Sc\langle v,v\rangle,\qquad v\in V.
\end{equation*}

\begin{rem}\label{rem_Two_scalar_products}
In Hilbert modules over $\mathbb{R}_d$, we consider two inner products that serve different purposes. The full inner product $\langle\cdot,\cdot\rangle$ is used for instance in the Riesz representation theorem, while the $\mathbb{R}$-linear scalar part $\Sc\langle\cdot,\cdot\rangle$ gives raise to the norm $\Vert v\Vert^2=\Sc\langle v,v\rangle$.
\end{rem}

For any Clifford Banach module $V$ we will denote the set of \textit{bounded right-linear operators} by
\begin{equation*}
\mathcal{B}(V):=\big\{T:V\rightarrow V\text{ right-linear}\;\big|\;\dom(T)=V,\,T\text{ is bounded}\big\},
\end{equation*}
and the set of \textit{closed right-linear operators} by
\begin{equation*}
\mathcal{K}(V):=\big\{T:V\rightarrow V\text{ right-linear}\;\big|\;\dom(T)\subseteq V\text{ is right-linear},\,T\text{ is closed}\big\}.
\end{equation*}
In difference to complex Banach spaces, in Cliffordian Banach modules, the spectrum is connected to the bounded invertibility of the second order operator \eqref{Eq_Qs}, which suggests the following definition of $S$-spectrum.

\begin{defi}[$S$-Spectrum]
For every $T\in\mathcal{K}(V)$, let us define the \textit{$S$-resolvent set} and the \textit{$S$-spectrum}
\begin{equation}\label{Eq_S_spectrum}
\rho_S(T):=\big\{s\in\mathbb{R}^{d+1}\;\big|\;Q_s[T]^{-1}\in\mathcal{B}(V)\big\}\qquad\text{and}\qquad\sigma_S(T):=\mathbb{R}^{d+1}\setminus\rho_S(T).
\end{equation}
Moreover, for every $s\in\rho_S(T)$ we define the \textit{left} and the \textit{right $S$-resolvent operator}
\begin{equation}\label{Eq_SL_SR}
S_L^{-1}(s,T):=Q_s[T]^{-1}\overline{s}-TQ_s[T]^{-1}\qquad\text{and}\qquad S_R^{-1}(s,T):=(\overline{s}-T)Q_s[T]^{-1}.
\end{equation}
\end{defi}

Although the $S$-resolvent set is defined via the invertibility of the second order operator $Q_s[T]$ defined on $\dom(T^2)$, we can use the right multiplication operator $\mathcal{I}^Rs$ from \eqref{Eq_IRs} to characterize it via the invertibility of a first order operator, defined on $\dom(T)$.

\begin{thm}\label{thm_rhoS_via_Is}
For every $T\in\mathcal{K}(V)$, the $S$-resolvent set \eqref{Eq_S_spectrum} can be characterized by
\begin{equation}\label{Eq_rhoS_via_Is}
\rho_S(T)=\big\{s\in\mathbb{R}^{d+1}\;\big|\;(T-\mathcal{I}^Rs)^{-1}\in\mathcal{B}(V)\big\},
\end{equation}
and for every $v\in V$ the respective resolvents are connected via
\begin{equation}\label{Eq_Resolvent_via_Is}
Q_s[T]^{-1}=(T-\mathcal{I}^R\overline{s})^{-1}(T-\mathcal{I}^Rs)^{-1}.
\end{equation}
\end{thm}

\begin{proof}
For the inclusion $\text{\grqq}\subseteq\text{\grqq}$ in \eqref{Eq_rhoS_via_Is} let $s\in\rho_S(T)$, i.e. $Q_s[T]^{-1}\in\mathcal{B}(V)$. Then there holds for every $v\in\dom(T)$
\begin{align}
(T-\mathcal{I}^R\overline{s})Q_s[T]^{-1}(T-\mathcal{I}^Rs)v&=(T-\mathcal{I}^R\overline{s})Q_s[T]^{-1}(Tv-vs) \notag \\
&=(T-\mathcal{I}^R\overline{s})(TQ_s[T]^{-1}v-Q_s[T]^{-1}vs) \notag \\
&=T^2Q_s[T]^{-1}v-TQ_s[T]^{-1}vs-TQ_s[T]^{-1}v\overline{s}+Q_s[T]^{-1}vs\overline{s} \notag \\
&=(T^2-2s_0T+|s|^2)Q_s[T]^{-1}v=v. \label{Eq_rhoS_via_Is_1}
\end{align}
Moreover, for every $v\in V$, there is
\begin{align}
(T-\mathcal{I}^Rs)(T-\mathcal{I}^R\overline{s})Q_s[T]^{-1}v&=\big(T^2-(\mathcal{I}^Rs)T-T(\mathcal{I}^R\overline{s})+(\mathcal{I}^Rs)(\mathcal{I}^R\overline{s})\big)Q_s[T]^{-1}v \notag \\
&=T^2Q_s[T]^{-1}v-TQ_s[T]^{-1}vs-TQ_s[T]^{-1}v\overline{s}+Q_s[T]^{-1}v\overline{s}s \notag \\
&=(T^2-2s_0T+|s|^2)Q_s[T]^{-1}v=v. \label{Eq_rhoS_via_Is_2}
\end{align}
From \eqref{Eq_rhoS_via_Is_1} it now follows that $T-\mathcal{I}^Rs$ is injective, and from \eqref{Eq_rhoS_via_Is_2} that $T-\mathcal{I}^Rs$ is surjective. Altogether we know that $T-\mathcal{I}^Rs$ is bijective and by the closed graph theorem also $(T-\mathcal{I}^Rs)^{-1}\in\mathcal{B}(V)$. \medskip

For the inverse inclusion $\text{\grqq}\supseteq\text{\grqq}$ in \eqref{Eq_rhoS_via_Is}, let $s\in\mathbb{R}^{d+1}$ with $(T-\mathcal{I}^Rs)^{-1}\in\mathcal{B}(V)$. In the \textit{first step} we show that also $(T-\mathcal{I}^R\overline{s})^{-1}\in\mathcal{B}(V)$. To do so, let us choose $J\in\mathbb{S}$ with $s\in\mathbb{C}_J$, and another imaginary unit $I\in\mathbb{S}$, with $IJ=-JI$, possible because of \cite[Proposition 2.2.10]{ColomboSabadiniStruppa2011}. We then consider the $\mathbb{R}$-linear operator
\begin{equation*}
R_s(v):=-(T-\mathcal{I}^Rs)^{-1}(vI)I,\qquad v\in V.
\end{equation*}
Then for every $v\in\dom(T)$, there is
\begin{align}
R_s((T-\mathcal{I}^R\overline{s})v)&=R_s(Tv-v\overline{s})=-(T-\mathcal{I}^Rs)^{-1}\big((Tv-v\overline{s})I\big)I \notag \\
&=-(T-\mathcal{I}^Rs)^{-1}\big((T-\mathcal{I}^Rs)(vI)\big)I=-(vI)I=v, \label{Eq_rhoS_via_Is_3}
\end{align}
where in the third equation we used that due to $IJ=-JI$ there is $\overline{s}I=Is$ and  hence $v\overline{s}I=vIs=\mathcal{I}^Rs(vI)$. Moreover, for every $v\in V$ we also have
\begin{align}
(T-\mathcal{I}^R\overline{s})R_s(v)&=TR_s(v)-R_s(v)\overline{s}=-T(T-\mathcal{I}^Rs)^{-1}(vI)I+(T-\mathcal{I}^Rs)^{-1}(vI)I\overline{s} \notag \\
&=(-T+\mathcal{I}^Rs)\big((T-\mathcal{I}^Rs)^{-1}(vI)\big)I=-(vI)I=v, \label{Eq_rhoS_via_Is_4}
\end{align}
where in the second line we used $Is\overline{s}=sI$, which is again due to $IJ=-JI$. Now, from the identity \eqref{Eq_rhoS_via_Is_3} it follows that $T-\mathcal{I}^R\overline{s}$ is injective and from \eqref{Eq_rhoS_via_Is_4} that $T-\mathcal{I}^R\overline{s}$ is surjective. Hence $T-\mathcal{I}^Rs$ is bijective and by the closed graph theorem then $(T-\mathcal{I}^R\overline{s})^{-1}\in\mathcal{B}(V)$. \medskip

In the \textit{second step}, we derive for every $v\in\dom(T^2)$ the identity
\begin{align*}
(T-\mathcal{I}^Rs)(T-\mathcal{I}^R\overline{s})v&=(T-\mathcal{I}^Rs)(Tv-v\overline{s})=T^2v-Tvs-Tv\overline{s}+v\overline{s}s \\
&=(T^2-2s_0T+|s|^2)v=Q_s[T]v.
\end{align*}
Hence $Q_s[T]=(T-\mathcal{I}^Rs)(T-\mathcal{I}^R\overline{s})$ is the product of two bijective operators, and hence bijective itself. This shows that $s\in\rho_S(T)$, and that
\begin{equation*}
Q_s[T]^{-1}=(T-\mathcal{I}^R\overline{s})^{-1}(T-\mathcal{I}^Rs)^{-1}. \qedhere
\end{equation*}
\end{proof}

\begin{rem}
Note that although the right hand side of the resolvent identity \eqref{Eq_Resolvent_via_Is} consists of the product of two only $\mathbb{R}$-linear operators, their product equals the pseudo resolvent $Q_s[T]^{-1}$ and hence is right-linear again.
\end{rem}

\begin{rem}
In the complex setting, the resolvent condition $(T-\lambda)^{-1}\in\mathcal{B}(V)$ is equivalent to $T-\lambda$ being bijective. This is due to the closed graph theorem, which concludes the boundedness of the inverse operator automatically. This is true also in the Clifford case for the operator $Q_s[T]$. Indeed, the proof of Theorem~\ref{thm_rhoS_via_Is} tells, if $Q_s[T]$ is bijective, then also $T-\mathcal{I}^Rs$ and $T-\mathcal{I}^R\overline{s}$ are both bijective and their inverses are bounded by the closed graph theorem. Hence also
\begin{equation*}
Q_s[T]^{-1}=(T-\mathcal{I}^Rs)^{-1}(T-\mathcal{I}^R\overline{s})^{-1}\in\mathcal{B}(V)
\end{equation*}
is a bounded operator. This shows, the bijectivity of $Q_s[T]$ alone already gives the boundedness of its inverse. Recall that the topology of $V$ as a Banach module over $\mathbb{R}_d$ and as a real Banach space is the same. Hence the closed graph theorem applies to right-linear and $\mathbb{R}$-linear operators both.
\end{rem}

The next proposition however states that if $Q_q[T]^{-1}\in\mathcal{B}(V)$ for one point $q\in\mathbb{R}^{d+1}$, then $Q_s[T]$ is closed for every $s\in\mathbb{R}^{d+1}$.

\begin{prop}
Let $T\in\mathcal{K}(V)$ with $\rho_S(T)\neq\emptyset$. Then
\begin{equation*}
Q_s[T]\text{ is closed for every }s\in\mathbb{R}^{d+1}.
\end{equation*}
\end{prop}

\begin{proof}
By assumption, there exists some $q\in\rho_S(T)$. With the corresponding operator $Q_q[T]$ we can rewrite the operator $Q_s[T]$, for every $s\in\mathbb{R}^{d+1}$, as
\begin{align}
Q_s[T]&=Q_q[T]+2(q_0-s_0)T+|s|^2-|q|^2 \notag \\
&=Q_q[T]\big(1+2(q_0-s_0)TQ_q[T]^{-1}+(|s|^2-|q|^2)Q_q[T]^{-1}\big). \label{Eq_Qs_closed_1}
\end{align}
Next note that $Q_q[T]^{-1}\in\mathcal{B}(V)$ due to $q\in\rho_S(T)$, and $T$ is closed, which means that the product $TQ_q[T]^{-1}$ is everywhere defined and closed. Hence $TQ_q[T]^{-1}\in\mathcal{B}(V)$ by the closed graph theorem. Altogether, this shows that the operator in the large bracket on the right hand side of \eqref{Eq_Qs_closed_1} is a bounded operator. Since also $Q_q[T]$ is closed (because $Q_q[T]^{-1}$ is closed), the representation \eqref{Eq_Qs_closed_1} gives a decomposition of $Q_s[T]$ into a closed and a bounded operator. So $Q_s[T]$ is closed as a consequence.
\end{proof}

The next proposition states that for a densely defined operator $T$ with nonempty resolvent set, also the operator $Q_s[T]$ is densely defined.

\begin{prop}
Let $T\in\mathcal{K}(V)$ with $\overline{\dom}(T)=V$ and $\rho_S(T)\neq\emptyset$. Then
\begin{equation*}
\overline{\dom}(T^2)=V.
\end{equation*}
In particular, $Q_s[T]$ is densely defined for every $s\in\mathbb{R}^{d+1}$.
\end{prop}

\begin{proof}
Let us start with $v\in\dom(T)$. Since $\overline{\dom}(T)=V$ by assumption, there exists a sequence $(w_n)_n\in\dom(T)$, such that
\begin{equation*}
\lim\limits_{n\rightarrow\infty}w_n=Tv.
\end{equation*}
Let us now choose some arbitrary $q\in\rho_S(T)$. Then $TQ_q[T]^{-1}\in\mathcal{B}(V)$, and there also converges
\begin{equation}\label{Eq_Qs_densely_defined_1}
\lim\limits_{n\rightarrow\infty}TQ_q[T]^{-1}w_n=TQ_q[T]^{-1}Tv.
\end{equation}
Since $v\in\dom(T)$, we can rewrite it as
\begin{equation*}
v=Q_q[T]Q_q[T]^{-1}v=TQ_q[T]^{-1}Tv-2q_0Q_q[T]^{-1}Tv+|q|^2Q_q[T]^{-1}v.
\end{equation*}
Plugging in the limit \eqref{Eq_Qs_densely_defined_1} for the first term on the right hand side, gives
\begin{align*}
v&=\lim\limits_{n\rightarrow\infty}TQ_q[T]^{-1}w_n-2q_0Q_q[T]^{-1}Tv+|q|^2Q_q[T]^{-1}v \\
&=\lim\limits_{n\rightarrow\infty}Q_q[T]^{-1}(Tw_n-2q_0Tv+|q|^2v).
\end{align*}
This proves that $v\in\overline{\ran}(Q_q[T]^{-1})=\overline{\dom}(T^2)$, and since $v\in\dom(T)$ was arbitrary, we get
\begin{equation*}
\dom(T)\subseteq\overline{\dom}(T^2).
\end{equation*}
However, since $\overline{\dom}(T)=V$ is dense, this inclusion shows that also $\overline{\dom}(T^2)=V$.
\end{proof}

\section{The adjoint of a Clifford operator}

In this section let $V$ be a Hilbert module over the Clifford algebra $\mathbb{R}_d$. We will introduce the adjoint operator $T^*$ of a densely defined unbounded operator $T:\dom(T)\subseteq V\rightarrow V$, and collect some of its main properties. In particular, in Theorem~\ref{thm_Spectrum_Tstar} we prove that the $S$-spectra of $T^*$ and $T$ coincide. For further results see also \cite{ColKim}.

\begin{defi}\label{defi_Adjoint_operator}
Let $T:\dom(T)\subseteq V\rightarrow V$ be a densely defined $\mathbb{R}$-linear operator. Then we define the \textit{adjoint operator} $T^*:\dom(T^*) \subset V\rightarrow V$ as the unique operator with the property
\begin{equation}\label{Eq_Adjoint_operator}
\Sc\langle T^*w,v\rangle=\Sc\langle w,Tv\rangle,\qquad v\in\dom(T),w\in\dom(T^*),
\end{equation}
and with the domain
\begin{equation}\label{Eq_Adjoint_operator_domain}
\dom(T^*):=\big\{w\in V\;\big|\;\exists v_w\in V: \Sc\langle v_w,v\rangle=\Sc\langle w,Tv\rangle,\,v\in\dom(T)\big\}.
\end{equation}
\end{defi}

Then $T^*$ is again a closed, $\mathbb{R}$-linear operator. The next lemma proves that for right-linear operators $T$, the adjoint operator is again right-linear, and there exists a characterization of the action \eqref{Eq_Adjoint_operator} and the domain \eqref{Eq_Adjoint_operator_domain} in terms of the inner product $\langle\cdot,\cdot\rangle$ instead of its scalar part $\Sc\langle\cdot,\cdot\rangle$, see also Remark~\ref{rem_Two_scalar_products}.

\begin{lem}\label{lem_Right_linear_adjoint}
Let $T:\dom(T)\subseteq V\rightarrow V$ be a densely defined right-linear operator. Then the operator $T^*$ from Definition~\ref{defi_Adjoint_operator} is right-linear, its domain can be characterized by
\begin{equation}\label{Eq_Right_linear_adjoint_domain}
\dom(T^*)=\big\{w\in V\;\big|\;\exists v_w\in V: \langle v_w,v\rangle=\langle w,Tv\rangle,\,v\in\dom(T)\big\},
\end{equation}
and it satisfies the identity
\begin{equation}\label{Eq_Right_linear_adjoint}
\langle T^*w,v\rangle=\langle w,Tv\rangle,\qquad v\in\dom(T),\,w\in\dom(T^*).
\end{equation}
\end{lem}

\begin{proof}
For the inclusion $\text{\grqq}\subseteq\text{\grqq}$ in \eqref{Eq_Right_linear_adjoint_domain}, let $w\in\dom(T^*)$. Note that any component $s_A$ of a Clifford number $s=\sum_{A\in\mathcal{A}}s_Ae_A$ can be written as $s_A=\Sc(s\,\overline{e_A})$, see e.g. \cite[Eq.(2.2)]{CMS24}. This allows us to write for every $v\in\dom(T)$
\begin{align}
\langle T^*w,v\rangle&=\sum\limits_{A\in\mathcal{A}}\Sc\big(\langle T^*w,v\rangle\overline{e_A}\big)e_A=\sum\limits_{A\in\mathcal{A}}\Sc\langle T^*w,v\,\overline{e_A}\rangle e_A \notag \\
&=\sum\limits_{A\in\mathcal{A}}\Sc\langle w,Tv\,\overline{e_A}\rangle e_A=\sum\limits_{A\in\mathcal{A}}\Sc\big(\langle w,Tv\rangle\overline{e_A}\big)e_A=\langle w,Tv\rangle. \label{Eq_Right_linear_adjoint_1}
\end{align}
The identity \eqref{Eq_Right_linear_adjoint_1} now on the one hand proves that $w$ is contained in the right hand side of \eqref{Eq_Right_linear_adjoint_domain} and also proves the identity \eqref{Eq_Right_linear_adjoint}. \medskip

For the inclusion $\text{\grqq}\supseteq\text{\grqq}$ in \eqref{Eq_Right_linear_adjoint_domain}, let $w\in V$ such that there exists some $v_w\in V$ with
\begin{equation*}
\langle v_w,v\rangle=\langle w,Tv\rangle,\qquad v\in\dom(T).
\end{equation*}
Then there clearly also is
\begin{equation*}
\Sc\langle v_w,v\rangle=\Sc\langle w,Tv\rangle,\qquad v\in\dom(T),
\end{equation*}
and we have proven that $w\in\dom(T^*)$, as defined in \eqref{Eq_Adjoint_operator_domain}. \medskip

Finally, we want to verify that $T^*$ is right-linear. With the identity \eqref{Eq_Right_linear_adjoint}, the linearity \eqref{Eq_Inner_product_linearity} and the conjugation property \eqref{Eq_Inner_product_property} of the inner product, there follows for every $s\in\mathbb{R}_d$
\begin{equation*}
\langle T^*(w)s,v\rangle=\overline{s}\langle T^*w,v\rangle=\overline{s}\langle w,Tv\rangle=\overline{\langle Tv,w\rangle s}=\overline{\langle Tv,ws\rangle}=\langle ws,Tv\rangle,\qquad v\in\dom(T).
\end{equation*}
Hence there is $ws\in\dom(T^*)$ and $T^*(ws)=T^*(w)s$.
\end{proof}

The condition that the adjoint operator in Definition~\ref{defi_Adjoint_operator} is only defined for operators $T$ with dense domain, is crucial in the sense that otherwise the element $v_w$ in \eqref{Eq_Adjoint_operator_domain} is not unique, and $T^*w$ not well defined. However, especially in Section~\ref{sec_Hinfty_FC}, this assumption will often be not satisfied for the $H^\infty$-functional calculus. To circumvent this problem, we need a generalization of the concept of operators.

\begin{defi}[Multivalued operators]\label{defi_Multivalued_operators}
A right-linear subspace $T\subseteq V\times V$ of the product space, is called a \textit{multivalued operator}. A right-linear operator $T:\dom(T)\subseteq V\rightarrow V$ is considered as a multivalued operator in the sense that $T$ is identified with $\graph(T)$.
\end{defi}

Let now $T,S\subseteq V\times V$ be two right-linear multivalued operators, and $s\in\mathbb{R}_d$. Then the sum and product of multivalued operators is defined as
\begin{align*}
T+S:=&\big\{(v,w_1+w_2)\;\big|\;(v,w_1)\in T,\,(v,w_2)\in S\big\}, \\
sT:=&\big\{(v,sw)\;\big|\;(v,w)\in T\big\}, \\
T\circ S:=&\big\{(v,u)\in V\times V\;\big|\;\exists w\in V\text{ such that }(v,w)\in S\text{ and }(w,u)\in T\big\}.
\end{align*}
The great advantage of treating multivalued operators is that for every right-linear multivalued operator $T\subseteq V\times V$ there exists the corresponding closure
\begin{equation}\label{Eq_Closure_multivalued}
\overline{T}:=\big\{(v,w)\in V\times V\;\big|\;\exists(v_n,w_n)_n\in T\text{ such that }(v,w)=\lim_{n\rightarrow\infty}(v_n,w_n)\big\},
\end{equation}
as well as the corresponding adjoint multivalued operator
\begin{equation}\label{Eq_Adjoint_multivalued}
T^*:=\big\{(w,v)\in V\times V\;\big|\;\langle v,\widetilde{v}\rangle=\langle w,\widetilde{w}\rangle,\text{ for all }(\widetilde{v},\widetilde{w})\in T\big\}.
\end{equation}
Note, that in the case that $T$ is a densely defined right-linear operator, the definition \eqref{Eq_Adjoint_multivalued} of the adjoint multivalued operator coincides with the definition \eqref{Eq_Adjoint_operator} of the classical adjoint operator, if one identifies $T$ and $T^*$ with their respective graphs of course. \medskip

The next lemma states some more basic properties of the adjoint operator. The proofs follow the same arguments as in the complex case.

\begin{lem}\label{lem_Properties_Adjoint}
Let $T,S\subseteq V\times V$ be right-linear multivalued operators. Then there holds the following properties:

\begin{enumerate}
\item[i)] If $T$ is closable, then $\overline{\dom}(T^*)=V$, and $\overline{T}=(T^*)^*$ and $(\overline{T})^*=T^*$; \medskip
\item[ii)] If $T\in\mathcal{B}(V)$, then also $T^*\in\mathcal{B}(V)$ with $\Vert T^*\Vert=\Vert T\Vert$; \medskip
\item[iii)] If $T\subseteq S$, then $S^*\subseteq T^*$; \medskip
\item[iv)] \begin{tabbing} $(ST)^*\supseteq T^*S^*$, \=and if $S\in\mathcal{B}(V)$ then even $(ST)^*=T^*S^*$, \\
\>and if $S\in\mathcal{K}(V)$, $T\in\mathcal{B}(V)$ then $(ST)^*=\overline{T^*S^*}$; \end{tabbing}
\item[v)] $(T+S)^*\supseteq T^*+S^*$, and if either $T\in\mathcal{B}(V)$ or $S\in\mathcal{B}(V)$ then $(T+S)^*=T^*+S^*$; \medskip
\item[vi)] For closed operators $T\in\mathcal{K}(V)$, we have that $T$ bijective if and only if $T^*$ is bijective. In this case there is $(T^*)^{-1}=(T^{-1})^*$.
\end{enumerate}
\end{lem}

A more detailed introduction into multivalued operators
can be found in \cite[Section 1]{BHS20}. \medskip

Let us now come back to the $S$-spectrum, and start by proving that the one of $T$ and $T^*$ coincide. The following theorem is the $S$-spectrum analog of the well known connection between spectrum of an operator and its adjoint.

\begin{thm}\label{thm_Spectrum_Tstar}
Let $T\in\mathcal{K}(V)$ be densely defined. Then there is
\begin{equation}\label{Eq_Spectrum_Tstar}
\sigma_S(T^*)=\sigma_S(T).
\end{equation}
Moreover, for every $s\in\rho_S(T^*)$, the $S$-resolvent operators of $T^*$ are given by
\begin{equation}\label{Eq_Resolvents_adjoint}
S_L^{-1}(s,T^*)=S_R^{-1}(\overline{s},T)^*\qquad\text{and}\qquad S_R^{-1}(s,T^*)=S_L^{-1}(\overline{s},T)^*.
\end{equation}
\end{thm}

\begin{proof}
Instead of \eqref{Eq_Spectrum_Tstar}, it is equivalent to prove $\rho_S(T^*)=\rho_S(T)$. For the first inclusion $\text{\grqq}\supseteq\text{\grqq}$, let $s\in\rho_S(T)$. From Theorem~\ref{thm_rhoS_via_Is} we know that $(T-\mathcal{I}^Rs)^{-1}\in\mathcal{B}(V)$. Since $\mathcal{I}^Rs$ is a bounded operator with adjoint given by $(\mathcal{I}^Rs)^*=\mathcal{I}^R\overline{s}$, we have $(T-\mathcal{I}^Rs)^*=T^*-\mathcal{I}^R\overline{s}$. From Lemma~\ref{lem_Properties_Adjoint} vi) we know that $(T-\mathcal{I}^Rs)^*$ is bijective, and hence there is also $T^*-\mathcal{I}^R\overline{s}$ bijective. It then follows again from Theorem~\ref{thm_rhoS_via_Is} that $\overline{s}\in\rho_S(T^*)$. Finally, since $\rho_S(T^*)$ is axially symmetric, we also get $s\in\rho_S(T^*)$. \medskip

The inverse inclusion $\text{\grqq}\subseteq\text{\grqq}$ follows from the first one applied to the operator $T^*$, and using that $(T^*)^*=T$, see Lemma~\ref{lem_Properties_Adjoint}~i). \medskip

Finally, we want to verify the relations \eqref{Eq_Resolvents_adjoint} of the $S$-resolvent operators. It follows immediately from Lemma~\ref{lem_Properties_Adjoint}~iv) and v) that $Q_s[T]^*\supseteq Q_s[T^*]$. However, since $s\in\rho_S(T^*)$, there is $\ran(Q_s[T^*])=V$, and $Q_s[T]^*=Q_s[T^*]$ is satisfied with equality. Consequently, there is also
\begin{equation*}
Q_s[T^*]^{-1}=(Q_S[T]^*)^{-1}=(Q_s[T]^{-1})^*
\end{equation*}
Moreover, for the operator $TQ_s[T]^{-1}$ this also implies
\begin{equation*}
(TQ_s[T]^{-1})^*\supseteq(Q_s[T]^{-1})^*T^*=Q_S[T^*]^{-1}T^*.
\end{equation*}
Since the left hand side is a closed operator, there also is
\begin{equation}\label{Eq_Resolvent_set_adjoint_4}
(TQ_s[T]^{-1})^*\supseteq\overline{Q_s[T^*]^{-1}T^*}.
\end{equation}
In order to show that the right hand side is everywhere defined, let $v\in V$. Since $T^*$ is densely defined by Lemma~\ref{lem_Properties_Adjoint}~i), there exists a sequence $(v_n)_n\in\dom(T^*)$ with $v=\lim_{n\rightarrow\infty}v_n$. Since moreover $T^*Q_s[T^*]^{-1}\in\mathcal{B}(V)$ there also converges
\begin{equation*}
T^*Q_s[T^*]^{-1}v=\lim\limits_{n\rightarrow\infty}T^*Q_s[T^*]^{-1}v_n=\lim\limits_{n\rightarrow\infty}Q_s[T^*]^{-1}T^*v_n.
\end{equation*}
Hence there is $v\in\dom(\overline{Q_s[T^*]^{-1}T^*})$ and
\begin{equation*}
\overline{Q_s[T^*]^{-1}T^*}v=T^*Q_s[T^*]^{-1}v.
\end{equation*}
This in particular means that the closure $\overline{Q_s[T^*]^{-1}T^*}$ is everywhere defined, and the inclusion \eqref{Eq_Resolvent_set_adjoint_4} turns into the equality
\begin{equation*}
(TQ_s[T]^{-1})^*=T^*Q_s[T^*]^{-1}.
\end{equation*}
With this identity, we then conclude the forms \eqref{Eq_Resolvents_adjoint} of the $S$-resolvent operators, namely
\begin{align*}
S_R^{-1}(\overline{s},T)^*&=\big((s-T)Q_s[T]^{-1}\big)^*=\big(sQ_s[T]^{-1}\big)^*-\big(TQ_s[T]^{-1}\big)^* \\
&=Q_s[T^*]^{-1}\overline{s}-T^*Q_s[T^*]^{-1}=S_L^{-1}(s,T^*),
\end{align*}
as well as
\begin{align*}
S_L^{-1}(\overline{s},T)^*&=\big(Q_s[T]^{-1}s-TQ_s[T]^{-1}\big)^*=\big(Q_s[T]^{-1}s\big)^*-\big(TQ_s[T]^{-1}\big)^* \\
&=\overline{s}Q_s[T^*]^{-1}-T^*Q_s[T^*]^{-1}=S_R^{-1}(s,T^*). \qedhere
\end{align*}
\end{proof}

We notice that some of the results of Theorem~\ref{thm_Spectrum_Tstar} have been explored in the quaternionic setting by other authors. For more details, see the various comments in \cite[Section 9.2]{CGK}, the recent paper \cite{Diogo} and the references therein.

\section{The $S$-functional calculus of the adjoint operator}\label{sec_SFC}

In this section let $V$ be a Hilbert module over the Clifford algebra $\mathbb{R}_d$. Since Theorem~\ref{thm_Spectrum_Tstar} proves that the $S$-spectra of $T$ and $T^*$ coincide, it is reasonable to also find a connection between the various $S$-functional calculi of $T$ and $T^*$. We will do this for the bounded $S$-functional calculus in Theorem~\ref{thm_Bounded_Tstar}, for the unbounded $S$-functional calculus in Theorem~\ref{thm_Unbounded_Tstar}, for the $\omega$-functional calculus in Theorem~\ref{thm_omega_SFC_Tstar} and for the $H^\infty$-functional calculus in Theorem~\ref{thm_Hinfty_Tstar}. To do so, we introduce for functions $f$, the corresponding function $f^\#$ as follows.

\begin{defi}
Let $U\subseteq\mathbb{R}^{d+1}$ be open and axially symmetric. Then for every $f:U\rightarrow\mathbb{R}_d$, we define the function
\begin{equation}\label{Eq_fsharp}
f^\#(s):=\overline{f(\overline{s})},\qquad s\in U.
\end{equation}
\end{defi}

The function $f^\#$ now admits the following basic holomorphicity properties.

\begin{lem}\label{lem_Properties_fsharp}
Let $U\subseteq\mathbb{R}^{d+1}$ be open and axially symmetric.

\begin{enumerate}
\item[i)] For every $f\in\mathcal{SH}_L(U)$, there is $f^\#\in\mathcal{SH}_R(U)$; \medskip
\item[ii)] For every $f\in\mathcal{SH}_R(U)$, there is $f^\#\in\mathcal{SH}_L(U)$; \medskip
\item[iii)] For every $f\in\mathcal{N}(U)$, there is $f^\#=f$.
\end{enumerate}
\end{lem}

\begin{proof}
i)\;\;Let $f_0,f_1:\mathcal{U}\rightarrow\mathbb{R}_d$ be the two components of the function $f$ according to the decomposition \eqref{Eq_Holomorphic_decomposition}. Then we can write the function $f^\#$ for every $(x,y)\in\mathcal{U}$, $J\in\mathbb{S}$ as
\begin{equation*}
f^\#(x+Jy)=\overline{f(x-Jy)}=\overline{f_0(x,y)-Jf_1(x,y)} =\overline{f_0(x,y)}+\overline{f_1(x,y)}J.
\end{equation*}
Hence the components of $f^\#$ are given by
\begin{equation*}
f_0^\#(u,v):=\overline{f_0(u,v)}\qquad\text{and}\qquad f_1^\#(u,v):=\overline{f_1(u,v)},\qquad(u,v)\in\mathcal{U}.
\end{equation*}
Since $f_0,f_1$ satisfy the symmetry conditions \eqref{Eq_Symmetry_condition} as well as the Cauchy-Riemann equations \eqref{Eq_Cauchy_Riemann_equations}, the same is true also for $f_0^\#,f_1^\#$. This means we have proven $f^\#\in\mathcal{SH}_R(U)$. \medskip

ii)\;\;Analogously to i). \medskip

iii)\;\;If $f$ is intrinsic, there clearly is
\begin{equation*}
f^\#(s)=\overline{f(\overline{s})}=\overline{\overline{f(s)}}=f(s),\qquad s\in U. \qedhere
\end{equation*}
\end{proof}

\subsection{The bounded $S$-functional calculus}\label{sec_Bounded_SFC}

In this subsection we derive the connection between the $S$-functional calculus of a bounded operator $T$ and its adjoint $T^*$. Let us start with the definition of the $S$-functional calculus of a bounded operator.

\begin{defi}[Bounded $S$-functional calculus]
Let $T\in\mathcal{B}(V)$. Then for every $f\in\mathcal{SH}_L(U')$ (resp. $f\in\mathcal{SH}_R(U')$), defined on some axially symmetric, open set $U'\supseteq\sigma_S(T)$, the \textit{bounded $S$-functional calculus} is defined as \medskip

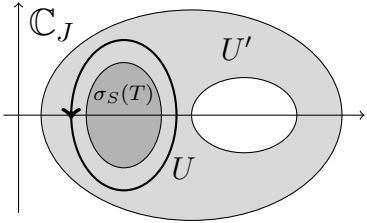
\begin{minipage}{0.33\textwidth}
\begin{center}
\begin{tikzpicture}
\draw[fill=black!15] (0,0) ellipse (2cm and 1.4cm) (0.6,0.6) node[anchor=south] {\large{$U'$}};
\draw (-2.3,1.2) node[anchor=west] {\Large{$\mathbb{C}_J$}};
\draw[fill=white] (0.7,0) ellipse (0.7cm and 0.5cm);
\draw[fill=black!30] (-0.9,0) ellipse (0.5cm and 0.7cm) node[anchor=south] {\scriptsize{$\sigma_S(T)$}};
\draw[thick] (-0.9,0) ellipse (0.7cm and 1cm);
\draw[ultra thick,->] (-1.6,0.05)--(-1.6,-0.05);
\draw (-0.4,-0.7) node[anchor=west] {\large{$U$}};
\draw[->] (-2.5,0)--(2.3,0);
\draw[->] (-2.3,-1.3)--(-2.3,1.5);
\end{tikzpicture}
\end{center}
\end{minipage}
\begin{minipage}{0.66\textwidth}
\begin{subequations}
\begin{align}
f(T):=&\frac{1}{2\pi}\int_{\partial U\cap\mathbb{C}_J}S_L^{-1}(s,T)ds_Jf(s), \label{Eq_Bounded_SFC_left} \\
\bigg(\text{resp. }f(T):=&\frac{1}{2\pi}\int_{\partial U\cap\mathbb{C}_J}f(s)ds_JS_R^{-1}(s,T)\bigg). \label{Eq_Bounded_SFC_right}
\end{align}
\end{subequations}
Here $J\in\mathbb{S}$ is arbitrary and $U$ is any axially symmetric, open set, with $\sigma_S(T)\subseteq U\subseteq\overline{U}\subseteq U'$, and regular enough boundary.
\end{minipage}

\medskip For more information on the bounded $S$-functional calculus we refer to \cite[Section 3.2]{CGK} or \cite[Chapter 3.3]{ColomboSabadiniStruppa2011}.
\end{defi}

Since we know from Lemma~\ref{lem_Properties_Adjoint}~ii) that there is $T^*\in\mathcal{B}(V)$ whenever $T\in\mathcal{B}(V)$, the following theorem gives a connection between the bounded $S$-functional calculus of $T$ and $T^*$.

\begin{thm}\label{thm_Bounded_Tstar}
Let $T\in\mathcal{B}(V)$ and $f\in\mathcal{SH}_L(U')$ or $f\in\mathcal{SH}_R(U')$ for some axially symmetric, open set $U'\supseteq\sigma_S(T)$. Then the bounded $S$-functional calculi of $T^*$ and $T$ are connected via
\begin{equation}\label{Eq_Bounded_Tstar}
f(T^*)=f^\#(T)^*.
\end{equation}
If moreover $f\in\mathcal{N}(U')$ is intrinsic, then there is even $f(T^*)=f(T)^*$.
\end{thm}

Note that if $f\in\mathcal{SH}_L(U')$, then $f(T^*)$ is understood as the left bounded $S$-functional calculus \eqref{Eq_Bounded_SFC_left}, and $f^\#(T)$ as the right bounded $S$-functional calculus \eqref{Eq_Bounded_SFC_right}. If on the other hand $f\in\mathcal{SH}_R(U')$, then $f^\#\in\mathcal{SH}_L(U')$, and the two sides of \eqref{Eq_Bounded_Tstar} are understood accordingly.

\begin{proof}[Proof of Theorem~\ref{thm_Bounded_Tstar}]
Without loss of generality let us only consider $f\in\mathcal{SH}_L(U')$. The case $f\in\mathcal{SH}_R(U')$ is analog. Using the definition \eqref{Eq_Bounded_SFC_right}, we write $f^\#(T)^*$ as
\begin{equation}\label{Eq_Bounded_Tstar_1}
f^\#(T)^*=\frac{1}{2\pi}\bigg(\int_{\partial U\cap\mathbb{C}_J}f^\#(s)ds_JS_R^{-1}(s,T)\bigg)^*=\frac{1}{2\pi}\int_{\partial U\cap\mathbb{C}_J}S_R^{-1}(s,T)^*d\overline{s_J}\,\overline{f^\#(s)}.
\end{equation}
By \eqref{Eq_Resolvents_adjoint} and \eqref{Eq_fsharp}, the individual terms of this integrand can now be written as
\begin{equation*}
S_R^{-1}(s,T)^*=S_L^{-1}(\overline{s},T^*),\qquad d\overline{s_J}=d\overline{s}_{-J},\qquad\text{and}\qquad\overline{f^\#(s)}=f(\overline{s}).
\end{equation*}
This reduces \eqref{Eq_Bounded_Tstar_1} to
\begin{equation}\label{Eq_Bounded_Tstar_2}
f^\#(T)^*=\frac{1}{2\pi}\int_{\partial U\cap\mathbb{C}_J}S_L^{-1}(\overline{s},T^*)d\overline{s}_{-J}f(\overline{s})=\frac{1}{2\pi}\int_{\partial U\cap\mathbb{C}_{-J}}S_L^{-1}(s,T^*)ds_{-J}f(s)=f(T^*),
\end{equation}
where in the second equation we used that the integral over $\overline{s}$ coincides with the one over $s$ since the integration path $\partial U\cap\mathbb{C}_J$ is symmetric with respect to the real line. In the last line we used that the $S$-functional calculus is independent of the chosen imaginary unit $J$ or $-J$. \medskip

Finally, for $f\in\mathcal{N}(U')$, there is $f=f^\#$ by Lemma~\ref{lem_Properties_fsharp}~iii), and $f(T^*)=f(T)^*$ is a consequence of \eqref{Eq_Bounded_Tstar}.
\end{proof}

\subsection{The unbounded $S$-functional calculus}\label{sec_Unbounded_SFC}

In comparison to the Section~\ref{sec_Bounded_SFC} we now consider unbounded, closed operators $T\in\mathcal{K}(V)$ and again look for a connection between the unbounded $S$-functional calculus of $T$ and $T^*$. We start with the definition of the $S$-functional calculus for arbitrary unbounded operators, and holomorphic functions which admit the same limit along any path $|s|\rightarrow\infty$. The idea is to "integrate around the spectrum" in the sense that the spectrum is completely outside of the integration path (which is the reason for the minus sign in front of the integral in \eqref{Eq_Unbounded_SFC}). Additionally, one has to add the constant value $f_\infty$ of the function $f$ at $\infty$ to the functional calculus.

\begin{defi}[Unbounded $S$-functional calculus]\label{defi_Unbounded_SFC}
Let $T\in\mathcal{K}(V)$ and $K\subseteq\rho_S(T)$ be compact, axially symmetric. Then for every $f\in\mathcal{SH}_L(\mathbb{R}^{d+1}\setminus K)$ (resp. $f\in\mathcal{SH}_R(\mathbb{R}^{d+1}\setminus K)$), for which the limit
\begin{equation*}
\lim\limits_{|s|\rightarrow\infty}f(s)=:f_\infty\in\mathbb{R}_d
\end{equation*}
exists, the \textit{unbounded $S$-functional calculus} is defined as \medskip

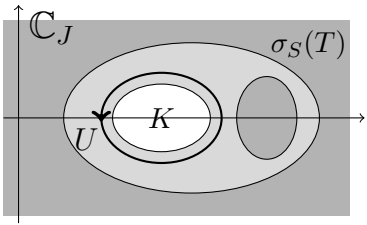
\begin{minipage}{0.33\textwidth}
\begin{center}
\begin{tikzpicture}
\fill[black!30] (-2.5,-1.3)--(-2.5,1.3)--(2.1,1.3)--(2.1,-1.3);
\draw (2.2,1.3) node[anchor=north east] {$\sigma_S(T)$};
\draw[fill=black!15] (0,0) ellipse (1.7cm and 1cm);
\draw[fill=black!30] (1,0) ellipse (0.4cm and 0.55cm);
\draw (-2.3,1.2) node[anchor=west] {\Large{$\mathbb{C}_J$}};
\draw[fill=white] (-0.4,0) ellipse (0.65cm and 0.45cm) node[anchor=center] {$K$};
\draw[thick] (-0.4,0) ellipse (0.8cm and 0.6cm);
\draw[ultra thick,->] (-1.2,0.05)--(-1.2,-0.05);
\draw (-1.1,0) node[anchor=north east] {\large{$U$}};
\draw[->] (-2.5,0)--(2.3,0);
\draw[->] (-2.3,-1.4)--(-2.3,1.5);
\end{tikzpicture}
\end{center}
\end{minipage}
\begin{minipage}{0.66\textwidth}
\begin{subequations}\label{Eq_Unbounded_SFC}
\begin{align}
f(T):=&f_\infty-\frac{1}{2\pi}\int_{\partial U\cap\mathbb{C}_J}S_L^{-1}(s,T)ds_Jf(s), \label{Eq_Unbounded_SFC_left} \\
\bigg(\text{resp. }f(T):=&f_\infty-\frac{1}{2\pi}\int_{\partial U\cap\mathbb{C}_J}f(s)ds_JS_R^{-1}(s,T)\bigg). \label{Eq_Unbounded_SFC_right}
\end{align}
\end{subequations}
Here $J\in\mathbb{S}$ is arbitrary and $U$ is an axially symmetric, open set, with $K\subseteq U\subseteq\overline{U}\subseteq\rho_S(T)$, and regular enough boundary.
\end{minipage}

\medskip For more information on the unbounded $S$-functional calculus we refer to \cite[Section 3.4]{FJBOOK} or \cite[Section 3.7]{ColomboSabadiniStruppa2011}.
\end{defi}

The following theorem now connects the unbounded $S$-functional calculus of $T^*$ and $T$.

\begin{thm}\label{thm_Unbounded_Tstar}
Let $T\in\mathcal{K}(V)$ be densely defined and $K\subseteq\rho_S(T)$ compact, axially symmetric. Then for every $f\in\mathcal{SH}_L(\mathbb{R}^{d+1}\setminus K)$ or $f\in\mathcal{SH}_R(\mathbb{R}^{d+1}\setminus K)$, for which the limit $\lim_{|s|\rightarrow\infty}f(s)$ exists in $\mathbb{R}_d$, the unbounded $S$-functional calculus of $T^*$ and $T$ are connected via
\begin{equation}\label{Eq_Unbounded_Tstar}
f(T^*)=f^\#(T)^*.
\end{equation}
If moreover $f\in\mathcal{N}(\mathbb{R}^{d+1}\setminus K)$ is intrinsic, then there is even $f(T^*)=f(T)^*$.
\end{thm}

Note that if $f\in\mathcal{SH}_L(\mathbb{R}^{d+1}\setminus K)$, then $f(T^*)$ is understood as the left unbounded $S$-functional calculus \eqref{Eq_Unbounded_SFC_left}, and $f^\#(T)$ as the right unbounded $S$-functional calculus \eqref{Eq_Unbounded_SFC_right}. If on the other hand $f\in\mathcal{SH}_R(\mathbb{R}^{d+1}\setminus K)$, then $f^\#\in\mathcal{SH}_L(\mathbb{R}^{d+1}\setminus K)$, and the two sides of \eqref{Eq_Unbounded_Tstar} are understood accordingly.

\begin{proof}[Proof of Theorem~\ref{thm_Unbounded_Tstar}]
Without loss of generality let us only consider $f\in\mathcal{SH}_L(\mathbb{R}^{d+1}\setminus K)$. The case $f\in\mathcal{SH}_R(\mathbb{R}^{d+1}\setminus K)$ is analog. First, since $\rho_S(T^*)=\rho_S(T)$ by Theorem~\ref{thm_Spectrum_Tstar}, we can choose the same compact set $K$ for $T$ and $T^*$ in Definition~\ref{defi_Unbounded_SFC}, i.e. the operator $f(T^*)$ on the left hand side of \eqref{Eq_Unbounded_Tstar} is well defined. Moreover, since by assumption the limit $\lim_{|s|\rightarrow\infty}f(s)=:f_\infty\in\mathbb{R}_d$ exists, there also exists the limit
\begin{equation*}
f_\infty^\#:=\lim\limits_{|s|\rightarrow\infty}f^\#(s)=\lim\limits_{|s|\rightarrow\infty}\overline{f(\overline{s})}=\overline{f_\infty}\in\mathbb{R}_d.
\end{equation*}
Hence, the function $f^\#$ satisfies the assumptions of Definition~\ref{defi_Unbounded_SFC} for the operator $T$, i.e. also $f^\#(T)$ on the right hand side of \eqref{Eq_Unbounded_Tstar} is well defined. \medskip

In order to show that $f(T^*)$ coincides with the adjoint of $f^\#(T)$, we first derive in the same way as in \eqref{Eq_Bounded_Tstar_1} and \eqref{Eq_Bounded_Tstar_2} the adjoint of the integral
\begin{equation*}
\bigg(\int_{\partial U\cap\mathbb{C}_J}f^\#(s)ds_JS_R^{-1}(s,T)\bigg)^*=\int_{\partial U\cap\mathbb{C}_{-J}}S_L^{-1}(s,T^*)ds_{-J}f(s).
\end{equation*}
This identity then shows that
\begin{align*}
f^\#(T)^*&=\bigg(f_\infty^\#-\frac{1}{2\pi}\int_{\partial U\cap\mathbb{C}_J}f^\#(s)ds_JS_R^{-1}(s,T)\bigg)^* \\
&=\overline{f_\infty^\#}-\frac{1}{2\pi}\bigg(\int_{\partial U\cap\mathbb{C}_J}f^\#(s)ds_JS_R^{-1}(s,T)\bigg)^* \\
&=f_\infty-\frac{1}{2\pi}\int_{\partial U\cap\mathbb{C}_{-J}}S_L^{-1}(s,T^*)ds_{-J}f(s)=f(T^*),
\end{align*}
where in the last equation we again used that the imaginary unit $J$ in the definition \eqref{Eq_Unbounded_SFC_left} of the functional calculus is arbitrary, and we are free to use $-J$ instead of $J$. \medskip

Finally, for $f\in\mathcal{N}(\mathbb{R}^{d+1}\setminus K)$, there is $f=f^\#$ by Lemma~\ref{lem_Properties_fsharp}~iii), and $f(T^*)=f(T)^*$ is a consequence of \eqref{Eq_Unbounded_Tstar}.
\end{proof}

\subsection{The $\omega$-functional calculus}

Since for unbounded operators $T\in\mathcal{K}(V)$ in Section~\ref{sec_Unbounded_SFC}, the number of functions $f$ which satisfy the assumptions of the unbounded $S$-functional calculus in Definition~\ref{defi_Unbounded_SFC} is very restricted, this section and the upcoming Section~\ref{sec_Hinfty_FC} give a more sophisticated approach to the $S$-functional calculus of unbounded operators. However in order to "integrate around the spectrum", we need some geometrical assumptions on the $S$-spectrum and some decay assumptions on the $S$-resolvent operators. There are different ways of doing this, and in this article we will consider the important special class of bisectorial operators. In order to introduce them, we define for every $\omega\in(0,\frac{\pi}{2})$ the open \textit{double sector}
\begin{equation}\label{Eq_Domega}
D_\omega:=\big\{s\in\mathbb{R}^{d+1}\setminus\{0\}\;\big|\;\Arg(s)\in I_\omega\big\},
\end{equation}
using the union of intervals $I_\omega:=(-\omega,\omega)\cup(\pi-\omega,\pi+\omega)$.

\begin{defi}[Bisectorial operators]\label{defi_Bisectorial_operators}
An operator $T\in\mathcal{K}(V)$ is called \textit{bisectorial} of angle $\omega\in(0,\frac{\pi}{2})$, if its $S$-spectrum is contained in the closed double sector
\begin{equation*}
\sigma_S(T)\subseteq\overline{D_\omega},
\end{equation*}
and for every $\varphi\in(\omega,\frac{\pi}{2})$ there exists some $C_\varphi\geq 0$, such that the left $S$-resolvent operator \eqref{Eq_SL_SR} is bounded by
\begin{equation}\label{Eq_SL_estimate}
\Vert S_L^{-1}(s,T)\Vert\leq\frac{C_\varphi}{|s|},\qquad s\in\mathbb{R}^{d+1}\setminus(D_\varphi\cup\{0\}).
\end{equation}
\end{defi}

\begin{rem}
The instance that for bisectorial operators only an upper bound on the left $S$-resolvent operator $S_L^{-1}(s,T)$ is assumed, is not a restriction to left holomorphic functions of the corresponding functional calculus. Indeed, it is proven in \cite[Lemma~2.12]{MS24}, that it already follows from \eqref{Eq_SL_estimate} that also the right $S$-resolvent operator $S_R^{-1}(s,T)$ is bounded by
\begin{equation}\label{Eq_SR_estimate}
\Vert S_R^{-1}(s,T)\Vert\leq\frac{2C_\varphi}{|s|},\qquad s\in\mathbb{R}^{d+1}\setminus(D_\varphi\cup\{0\}).
\end{equation}
\end{rem}

In order to make the integrals in Definition~\ref{defi_omega_SFC} of the $\omega$-functional calculus converge, we also have to assume certain integrability assumptions on the function $f$ of the functional calculus.

\begin{defi}\label{defi_SH_0}
For every $\theta\in(0,\frac{\pi}{2})$ let $D_\theta$ be the double sector and $I_\theta$ the union of intervals from \eqref{Eq_Domega}. Then we define the function spaces
\begin{align*}
\text{i)}\;\;&\mathcal{SH}_L^0(D_\theta):=\bigg\{f\in\mathcal{SH}_L(D_\theta)\;\bigg|\;f\text{ is bounded},\;\int_0^\infty|f(te^{J\phi})|\frac{dt}{t}<\infty,\;J\in\mathbb{S},\phi\in I_\theta\bigg\}, \\
\text{ii)}\;\;&\mathcal{SH}_R^0(D_\theta):=\bigg\{f\in\mathcal{SH}_R(D_\theta)\;\bigg|\;f\text{ is bounded},\;\int_0^\infty|f(te^{J\phi})|\frac{dt}{t}<\infty,\;J\in\mathbb{S},\phi\in I_\theta\bigg\}, \\
\text{iii)}\;\;&\mathcal{N}^0(D_\theta):=\bigg\{f\in\mathcal{N}(D_\theta)\;\bigg|\;f\text{ is bounded},\;\int_0^\infty|f(te^{J\phi})|\frac{dt}{t}<\infty,\;J\in\mathbb{S},\phi\in I_\theta\bigg\}.
\end{align*}
\end{defi}

The idea of the upcoming $\omega$-functional calculus is now to "integrate around the spectrum", i.e. around the left and right sectors, in the sense that closing the path on the left and on the right is not necessary. Indeed, due to \cite[Lemma~3.2]{MS24}, the integrability assumptions in Definition~\ref{defi_SH_0} imply that $f$ vanishes at $\infty$ inside the double sector. Moreover, integrating through the point $s=0$, which may be a point in the $S$-spectrum, is not allowed. So the integrand in this point is understood as the limit of the integration path approaching zero from the different directions. Also this is possible since $f$ vanishes as $s\rightarrow 0$ by \cite[Lemma~3.2]{MS24}.

\begin{defi}[$\omega$-functional calculus]\label{defi_omega_SFC}
Let $T\in\mathcal{K}(V)$ be bisectorial of angle $\omega\in(0,\frac{\pi}{2})$. Then for every $f\in\mathcal{SH}_L^0(D_\theta)$ (resp. $f\in\mathcal{SH}_R^0(D_\theta)$), the \textit{$\omega$-functional calculus} is defined as \medskip

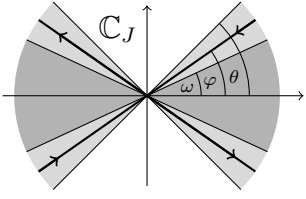
\begin{minipage}{0.35\textwidth}
\begin{center}
\begin{tikzpicture}[scale=0.8]
\fill[black!15] (0,0)--(1.56,1.56) arc (45:-45:2.2)--(0,0)--(-1.56,-1.56) arc (225:135:2.2);
\fill[black!30] (0,0)--(2,0.93) arc (25:-25:2.2)--(0,0)--(-2,-0.93) arc (205:155:2.2);
\draw (2,0.93)--(-2,-0.93);
\draw (2,-0.93)--(-2,0.93);
\draw (1.56,1.56)--(-1.56,-1.56);
\draw (1.56,-1.56)--(-1.56,1.56);
\draw (0.9,0) arc (0:25:0.9) (0.67,-0.1) node[anchor=south] {\tiny{$\omega$}};
\draw (1.3,0) arc (0:35:1.3) (1.05,-0.05) node[anchor=south] {\tiny{$\varphi$}};
\draw (1.7,0) arc (0:45:1.7) (1.45,0.05) node[anchor=south] {\tiny{$\theta$}};
\draw[thick] (1.8,1.26)--(-1.8,-1.26);
\draw[thick] (1.8,-1.26)--(-1.8,1.26);
\draw[thick,->] (1.8,1.26)--(1.47,1.03);
\draw[thick,->] (0,0)--(1.47,-1.03);
\draw[thick,->] (-1.8,-1.26)--(-1.47,-1.03);
\draw[thick,->] (0,0)--(-1.47,1.03);
\draw[->] (-2.4,0)--(2.6,0);
\draw[->] (0,-1.5)--(0,1.5) node[anchor=north east] {\large{$\mathbb{C}_J$}};
\end{tikzpicture}
\end{center}
\end{minipage}
\begin{minipage}{0.64\textwidth}
\begin{subequations}\label{Eq_omega}
\begin{align}
f(T):=&\frac{1}{2\pi}\int_{\partial D_\varphi\cap\mathbb{C}_J}S_L^{-1}(s,T)ds_Jf(s), \label{Eq_omega_left} \\
\bigg(\text{resp. }f(T):=&\frac{1}{2\pi}\int_{\partial D_\varphi\cap\mathbb{C}_J}f(s)ds_JS_R^{-1}(s,T)\bigg). \label{Eq_omega_right}
\end{align}
\end{subequations}
Here, $J\in\mathbb{S}$ and $\varphi\in(\omega,\theta)$ are arbitrary and the integral is independent of the choice.
\end{minipage}
\end{defi}

For more information, on the $\omega$-functional calculus, we refer to \cite[Section 3]{MS24}. \medskip

In order to characterize the $\omega$-functional calculus $f(T^*)$ of the adjoint, we first have to verify that $T^*$ is indeed a bisectorial operator.

\begin{lem}\label{lem_Sectorial_Tstar}
Let $T\in\mathcal{K}(V)$ be bisectorial of angle $\omega\in(0,\frac{\pi}{2})$. Then $T$ is densely defined, and also $T^*\in\mathcal{K}(V)$ is bisectorial of the same angle $\omega$.
\end{lem}

\begin{proof}
First, it is shown in \cite[Theorem~3.3]{RIGHT}, that $T$ is densely defined, and hence $T^*$ a well defined closed operator. Moreover, by Theorem~\ref{thm_Spectrum_Tstar} the $S$-spectrum of $T^*$ is contained in the same closed double sector
\begin{equation*}
\sigma_S(T^*)=\sigma_S(T)\subseteq\overline{D_\omega}.
\end{equation*}
In order to verify the norm estimate \eqref{Eq_SL_estimate} of $S_L^{-1}(s,T^*)$, we use the representation \eqref{Eq_Resolvents_adjoint} of $S_L^{-1}(s,T^*)$, the operator norm of the adjoint operator in Lemma~\ref{lem_Properties_Adjoint}~ii), as well as the upper bound of the right resolvent operator \eqref{Eq_SR_estimate}, to get
\begin{equation*}
\Vert S_L^{-1}(s,T^*)\Vert=\Vert S_R^{-1}(\overline{s},T)^*\Vert=\Vert S_R^{-1}(\overline{s},T)\Vert\leq\frac{2C_\varphi}{|s|},\qquad s\in\mathbb{R}^{d+1}\setminus\big(D_\varphi\cup\{0\}\big). \qedhere
\end{equation*}
\end{proof}

\begin{thm}\label{thm_omega_SFC_Tstar}
Let $T\in\mathcal{K}(V)$ be bisectorial of angle $\omega\in(0,\frac{\pi}{2})$. Then for every $\theta\in(0,\frac{\pi}{2})$, and $f\in\mathcal{SH}_L^0(D_\theta)$ or $f\in\mathcal{SH}_R^0(D_\theta)$, the $\omega$-functional calculi of $T^*$ and $T$ are connected via
\begin{equation}\label{Eq_omega_SFC_Tstar}
f(T^*)=f^\#(T)^*.
\end{equation}
If moreover $f\in\mathcal{N}(D_\theta)$ is intrinsic, then there is even $f(T^*)=f(T)^*$.
\end{thm}

Note that if $f\in\mathcal{SH}_L(D_\theta)$, then $f(T^*)$ is understood as the left $\omega$-functional calculus \eqref{Eq_omega_left}, and $f^\#(T)$ as the right $\omega$-functional calculus \eqref{Eq_omega_right}. If on the other hand $f\in\mathcal{SH}_R(D_\theta)$, then $f^\#\in\mathcal{SH}_L(D_\theta)$, and the two sides of \eqref{Eq_omega_SFC_Tstar} are understood accordingly.

\begin{proof}[Proof of Theorem~\ref{thm_omega_SFC_Tstar}]
Since we have verified in Lemma~\ref{lem_Sectorial_Tstar} that $T^*$ is a bisectorial operator of angle $\omega$, the functional calculus $f(T^*)$ is well defined. Since it is clear by the Definition~\ref{defi_SH_0} of the space, that for every $f\in\mathcal{SH}_L(D_\theta)$ there is $f^\#\in\mathcal{SH}_R(D_\theta)$, and for every $f\in\mathcal{SH}_R(D_\theta)$ there is $f^\#\in\mathcal{SH}_L(D_\theta)$, also the functional calculus $f^\#(T)$ is well defined. The actual equality \eqref{Eq_omega_SFC_Tstar} then follows in the same manner as in \eqref{Eq_Bounded_Tstar_1} and \eqref{Eq_Bounded_Tstar_2}, namely
\begin{equation*}
f^\#(T)^*=\frac{1}{2\pi}\bigg(\int_{\partial D_\varphi\cap\mathbb{C}_J}f^\#(s)ds_JS_R^{-1}(s,T)\bigg)^*=\frac{1}{2\pi}\int_{\partial D_\varphi\cap\mathbb{C}_{-J}}S_L^{-1}(s,T^*)ds_{-J}f(s)=f(T^*).
\end{equation*}
Note that all the manipulations in \eqref{Eq_Bounded_Tstar_1} and \eqref{Eq_Bounded_Tstar_2} are also allowed in this setting of the integral along the unbounded path $\partial D_\theta\cap\mathbb{C}_J$. Indeed, carrying the adjoint inside the integral is allowed, since the map that associates bounded operators with its adjoint is continuous in the operator norm topology, and the integral converges in this operator norm topology. \medskip

Finally, for $f\in\mathcal{N}(D_\theta)$ there is $f=f^\#$ by Lemma~\ref{lem_Properties_fsharp}~iii), and $f(T^*)=f(T)^*$ follows immediately from \eqref{Eq_omega_SFC_Tstar}.
\end{proof}

\subsection{The $H^\infty$-functional calculus}\label{sec_Hinfty_FC}

The $H^\infty$-functional calculus, which we consider in this subsection, is a regularization procedure that extends the $\omega$-functional calculus of Definition~\ref{defi_omega_SFC}. In particular, it avoids the very restrictive integrability assumptions on $f$ of Definition~\ref{defi_SH_0}, and rather treats the polynomially growing functions from Definition~\ref{defi_SH_poly}. The main result will then be formally the same identity as in \eqref{Eq_omega_SFC_Tstar}, but interpreted as unbounded or even multivalued operators.

\begin{defi}\label{defi_SH_poly}
For every $\alpha\in \mathbb{N}$ and $\theta\in(0,\frac{\pi}{2})$ let $D_\theta$ be the double sector from \eqref{Eq_Domega}. Then we define the function spaces
\begin{align*}
\text{i)}\;\;&\mathcal{SH}_L^\poly(D_\theta):=\bigg\{f\in\mathcal{SH}_L(D_\theta)\;\bigg|\;\exists C,\alpha\geq 0: |f(s)|\leq C\Big(|s|^\alpha+\frac{1}{|s|^\alpha}\Big),\text{ for every }s\in D_\theta\bigg\}; \\
\text{ii)}\;\;&\mathcal{SH}_R^\poly(D_\theta):=\bigg\{f\in\mathcal{SH}_R(D_\theta)\;\bigg|\;\exists C,\alpha\geq 0: |f(s)|\leq C\Big(|s|^\alpha+\frac{1}{|s|^\alpha}\Big),\text{ for every }s\in D_\theta\bigg\}; \\
\text{iii)}\;\;&\mathcal{N}^\poly(D_\theta):=\bigg\{f\in\mathcal{N}(D_\theta)\;\bigg|\;\exists C,\alpha\geq 0: |f(s)|\leq C\Big(|s|^\alpha+\frac{1}{|s|^\alpha}\Big),\text{ for every }s\in D_\theta\bigg\}.
\end{align*}
\end{defi}

In contrast to the previous sections on the bounded, the unbounded and the $\omega$-functional calculus, the following definition of the $H^\infty$-functional calculus is different for left and for right slice hyperholomorphic functions.

\begin{defi}[$H^\infty$-functional calculus]\label{defi_Hinfty_SFC}
Let $T\in\mathcal{K}(V)$ be an injective bisectorial operator of angle $\omega\in(0,\frac{\pi}{2})$. Then for every $f\in\mathcal{SH}_L^\poly(D_\theta)$, for some $\theta\in(\omega,\frac{\pi}{2})$, we define the \textit{left $H^\infty$-functional calculus} as
\begin{equation}\label{Eq_Left_Hinfty}
f(T):=e(T)^{-1}(ef)(T),
\end{equation}
and for every $f\in\mathcal{SH}_R^\poly(D_\theta)$, for some $\theta\in(\omega,\frac{\pi}{2})$, the \textit{right $H^\infty$-functional calculus} as
\begin{equation}\label{Eq_Right_Hinfty}
f(T):=\overline{(fe)(T)e(T)^{-1}},
\end{equation}
where all the terms $e(T)$, $(ef)(T)$ and $(fe)(T)$ are understood as the $\omega$-functional calculus \eqref{Eq_omega}. In both cases, the regularizer function $e\in\mathcal{N}^0(D_\theta)$ is given by $e(s)=\frac{s^m}{(1+s^2)^m}$, for some $m\in\mathbb{N}$ with $m>\alpha$, and $\alpha$ from Definition~\ref{defi_SH_poly}. Also, since $T$ is injective, the operator $e(T)=T^m(1+T^2)^{-m}$ is injective as well, i.e. $e(T)^{-1}$ is an unbounded closed operator. Hence the left $H^\infty$-functional calculus \eqref{Eq_Left_Hinfty} is a closed operator (as the product of a closed and a bounded operator), while the right $H^\infty$-functional calculus \eqref{Eq_Right_Hinfty} is a multivalued operator (since $(fe)(T)e(T)^{-1}$ is in general not closable and the closure is understood in the sense of multivalued operators \eqref{Eq_Closure_multivalued}). For more details on the $H^\infty$-functional on the $H^\infty$-functional calculus we refer to \cite{RIGHT,MS24}.
\end{defi}

\begin{thm}\label{thm_Hinfty_Tstar}
Let $T\in\mathcal{K}(V)$ be injective and bisectorial of angle $\omega\in(0,\frac{\pi}{2})$. Then for every $f\in\mathcal{SH}_L^\poly(D_\theta)$ or $f\in\mathcal{SH}_R^\poly(D_\theta)$, for some $\theta\in(\omega,\frac{\pi}{2})$, the $H^\infty$-functional calculi of $T^*$ and $T$ are connected via
\begin{equation*}
f(T^*)=f^\#(T)^\star.
\end{equation*}
\end{thm}

Note, that if $f\in\mathcal{SH}_L^\poly(D_\theta)$, then $f(T^*)$ is understood as the left $H^\infty$-functional calculus \eqref{Eq_Left_Hinfty} and $f^\#(T)$ as the right $H^\infty$-functional calculus \eqref{Eq_Right_Hinfty}. Moreover, since $f^\#(T)$ is densely defined by \cite[Proposition 4.10]{RIGHT}, the adjoint $f^\#(T)^*$ is considered as the adjoint \eqref{Eq_Adjoint_operator} of a closed operator. Conversely, if $f\in\mathcal{SH}_R^\poly(D_\theta)$, then $f(T^*)$ is a multivalued operator as in \eqref{Eq_Right_Hinfty}, and since $f^\#(T)$ for the left holomorphic function $f^\#$ is not densely defined in general, its adjoint $f^\#(T)^*$ has to be considered in terms of multivalued operators \eqref{Eq_Adjoint_multivalued}.

\begin{proof}[Proof of Theorem~\ref{thm_Hinfty_Tstar}]
In order for $f(T^*)$ to be well defined, we have to check if $T^*$ satisfies the assumptions of Definition~\ref{defi_Hinfty_SFC}. The fact that $T^*$ is bisectorial of angle $\omega$ was already proven in Lemma~\ref{lem_Sectorial_Tstar}, and the injectivity follows from
\begin{equation*}
\ker(T^*)=\ran(T)^\perp=\{0\},
\end{equation*}
where the range of a injective, bisectorial operator is dense by \cite[Theorem~3.3~ii)]{RIGHT}. \medskip

In the first part of the proof we will treat the case $f\in\mathcal{SH}_L(D_\theta)$. Let $e$ be the regularizer of the function $f$ in Definition~\ref{defi_Hinfty_SFC}. Then the same function $e$ is also a regularizer of $f^\#$. From the definition of the $H^\infty$-functional calculus there then follows
\begin{equation*}
f(T^*)=e(T^*)^{-1}(ef)(T^*).
\end{equation*}
Using now that
\begin{equation}\label{Eq_Hinfty_Tstar_1}
e(T^*)=e(T)^*,\qquad\text{as well as}\qquad(ef)(T^*)=(f^\#e)(T)^*,
\end{equation}
by \eqref{Eq_omega_SFC_Tstar}, and also using the properties of Lemma~\ref{lem_Properties_Adjoint}~vi) and iv), we get
\begin{equation*}
f(T^*)=(e(T)^*)^{-1}(f^\#e)(T)^*=(e(T)^{-1})^*(f^\#e)(T)^*=\big((f^\#e)(T)e(T)^{-1}\big)^*.
\end{equation*}
Finally, using the property Lemma~\ref{lem_Properties_Adjoint}~i), transforms the right hand side the right $H^\infty$-functional calculus of $f^\#$, namely
\begin{equation*}
f(T^*)=\big(\overline{(f^\#e)(T)e(T)^{-1}}\big)^*=f^\#(T)^*.
\end{equation*}
In the second part of the proof we will treat the case $f\in\mathcal{SH}_R(D_\theta)$. Then the definition \eqref{Eq_Left_Hinfty} of the left $H^\infty$-functional for the function $f^\#$ gives
\begin{equation*}
\big(f^\#(T)\big)^*=\big(e(T)^{-1}(ef^\#)(T)\big)^*.
\end{equation*}
Now using the Lemma~\ref{lem_Properties_Adjoint}~iv), we can split up the adjoint of the product, and get
\begin{equation*}
\big(f^\#(T)\big)^*=\overline{(ef^\#)(T)^*(e(T)^{-1})^*}.
\end{equation*}
Then similarly to \eqref{Eq_Hinfty_Tstar_1}, we can then carry the adjoint on the operator, which leaves us with exactly the definition of the right $H^\infty$-functional calculus of $T^*$, namely
\begin{equation*}
\big(f^\#(T)\big)^*=\overline{(fe)(T^*)e(T^*)^{-1}}=f(T^*). \qedhere
\end{equation*}
\end{proof}

\section{The multiplication operator in the Clifford setting}\label{sec_Multiplication_operator}

Now that we have developed all the $S$-functional calculi of the adjoint operator in Section~\ref{sec_SFC}, we will introduce in this section the multiplication operator as an example.

\begin{defi}
Let $X$ be a measure space, and $h:X\rightarrow\mathbb{R}_d$ measurable. Then define the \textit{multiplication operator}
\begin{equation}\label{Eq_Multiplication_operator}
(M_hf)(x):=h(x)f(x)\qquad\text{with}\qquad\dom(M_h):=\big\{f\in L^2(X)\;\big|\;hf\in L^2(X)\big\}.
\end{equation}
\end{defi}

The following lemma provides some basic properties of the multiplication operator \eqref{Eq_Multiplication_operator}. Since the proofs are mainly the same as in the complex case, they will be omitted.

\begin{lem}\label{lem_Properties_multiplication_operator}
Let $X$ be a measure space, and $g,h:X\rightarrow\mathbb{R}_d$ be measurable. Then there is:

\begin{enumerate}
\item[i)] $M_h$ is a densely defined, closed operator; \medskip
\item[ii)] If $X$ is $\sigma$-finite, then $M_h\in\mathcal{B}(L^2(X))$ if and only if $h\in L^\infty(X)$. \\
In this case there is $\Vert h\Vert_{L^\infty}\leq\Vert M_h\Vert\leq 2^{\frac{d}{2}}\Vert h\Vert_{L^\infty}$; \medskip
\item[iii)] $M_gM_h\subseteq M_{gh}$,\quad with\quad $\dom(M_gM_h)=\dom(M_{gh})\cap\dom(M_h)$.
\end{enumerate}
\end{lem}

An, in particular for this paper, important observation is now that the adjoint of the multiplication operator is again a multiplication operator with the conjugate function.

\begin{thm}\label{thm_Adjoint_multiplication_operator}
Let $X$ be a measure space, and $h:X\rightarrow\mathbb{R}_d$ be measurable. Then there is
\begin{equation*}
(M_h)^*=M_{\overline{h}}.
\end{equation*}
\end{thm}

\begin{proof}
For the inclusion $\text{\grqq}\subseteq\text{\grqq}$ let $g\in\dom(M_h^*)$, i.e., by Lemma~\ref{lem_Right_linear_adjoint}, there is
\begin{equation*}
\langle M_h^*g,f\rangle=\langle g,M_hf\rangle,\qquad f\in\dom(M_h).
\end{equation*}
Plugging in the definition of $M_h$ and using the properties \eqref{Eq_Inner_product_property} of the inner product, we can rewrite the right hand side as
\begin{equation*}
\langle M_h^*g,f\rangle=\langle g,hf\rangle=\langle\overline{h}g,f\rangle,\qquad f\in\dom(M_h).
\end{equation*}
Since this is true for every $f$ in the dense subset $\dom(M_h)$, it follows that there coincides
\begin{equation*}
M_h^*g=\overline{h}g.
\end{equation*}
This equality now proves that $\overline{h}g\in L^2(X)$, i.e. $g\in\dom(M_{\overline{h}})$, as well as
\begin{equation*}
M_h^*g=M_{\overline{h}}g.
\end{equation*}
For the inverse inclusion $\text{\grqq}\supseteq\text{\grqq}$ let $g\in\dom(M_{\overline{h}})$, i.e $g\in L^2(X)$ and $\overline{h}g\in L^2(X)$. Then there holds the identity
\begin{equation*}
\langle M_{\overline{h}}g,f\rangle=\langle\overline{h}g,f\rangle=\langle g,hf\rangle=\langle g,M_hf\rangle,\qquad f\in\dom(M_h).
\end{equation*}
By the definition \eqref{Eq_Adjoint_operator} of the adjoint operator, this shows $g\in\dom(M_h^*)$ and $M_h^*g=M_{\overline{h}}g$.
\end{proof}

Next we define the following subset of the Clifford algebra $\mathbb{R}_d$
\begin{equation}\label{Eq_NRd}
\mathcal{N}(\mathbb{R}_d):=\big\{s\in\mathbb{R}_d\;\big|\;|s|^2=s\overline{s}=s\overline{s}\big\}.
\end{equation}
Note that every element $0\neq s\in\mathcal{N}(\mathbb{R}_d)$ is invertible, with inverse given by $s^{-1}=\frac{\overline{s}}{|s|^2}$. \medskip

If we now reduce ourselves to functions $h:X\rightarrow\mathcal{N}(\mathbb{R}^d)$, which map into this space \eqref{Eq_NRd}, the following theorem transfers from the complex into the Clifford setting with the same proof.

\begin{thm}\label{thm_Bijective}
Let $X$ be a $\sigma$-finite measure space, and $h:X\rightarrow\mathcal{N}(\mathbb{R}_d)$ be measurable.

\begin{enumerate}
\item[i)] $M_h$ is injective if and only of $h(x)\neq 0$ for a.e. $x\in X$. \\
In this case the inverse operator is given by $M_h^{-1}=M_{h^{-1}}$. \medskip
\item[ii)] $M_h$ is surjective if and only if there exists $\varepsilon>0$, such that $|h(x)|\geq\varepsilon$, for a.e. $x\in X$.
\end{enumerate}
\end{thm}

\begin{lem}\label{lem_QsMh}
Let $X$ be a measure space, $h:X\rightarrow\mathbb{R}^{d+1}$ be measurable. Then for every $s\in\mathbb{R}^{d+1}$, the operator $Q_s[M_h]$ from \eqref{Eq_Qs} can be written as
\begin{equation}\label{Eq_QsMh}
Q_s[M_h]=M_{h^2-2s_0h+|s|^2}.
\end{equation}
\end{lem}

\begin{proof}
First, since $h:X\rightarrow\mathbb{R}^{d+1}$ maps into set of paravectors, the real polynomial combination $h^2-2s_0h+|s|^2:X\rightarrow\mathcal{N}(\mathbb{R}_d)$ maps into the set $\mathcal{N}(\mathbb{R}^d)$ from \eqref{Eq_NRd}. \medskip

For the operator inclusion $M_h^2-2s_0M_h+|s|^2\subseteq M_{h^2-2s_0h+|s|^2}$, let $f\in\dom(M_h^2)$. This means that $h^2f,hf\in L^2(X)$, and consequently also $(h^2-2s_0h+|s|^2)f\in L^2(X)$. This proves that $f\in\dom(M_{h^2-2s_0h+|s|^2})$, and moreover, the equality of the operation actions $(M_h^2-2s_0M_h+|s|^2)f=M_{h^2-2s_0h+|s|^2}f$ is then trivial. \medskip

For the inverse inclusion $M_{h^2-2s_0h+|s|^2}\subseteq M_h^2-2s_0M_h+|s|^2$, let $f\in\dom(M_{h^2-2s_0h+|s|^2})$, i.e. $f\in L^2(X)$ and $(h^2-2s_0h+|s|^2)f\in L^2(X)$. Let us now choose the constants
\begin{equation*}
M:=\sup\limits_{q\in\mathbb{R}^{d+1}}\frac{|q|}{1+|q^2-2s_0q+|s|^2|}<\infty,
\end{equation*}
with which we can now estimate
\begin{equation*}
|h(x)f(x)|=|h(x)||f(x)|\leq M\big(1+|h(x)^2-2s_0h(x)+|s|^2|\big)|f(x)|,\qquad x\in X.
\end{equation*}
Since $f\in L^2(X)$ and $(h^2-2s_0h+|s|^2)f\in L^2(X)$, it follows that also $hf\in L^2(X)$. Consequently, there follows also $h^2f\in L^2$, and in particular $f\in\dom(M_h^2)$. \end{proof}

Next, we want to describe the $S$-spectrum of the multiplication operator $M_h$. The key object in this description will be the essential range of a function.

\begin{defi}[Essential range]
Let $X$ be a measure space. Then for every measurable function $h:X\rightarrow\mathbb{R}^{d+1}$, we define the \textit{essential range}
\begin{equation}\label{Eq_essran}
\essran(h):=\Big\{s\in\mathbb{R}^{d+1}\;\Big|\;\forall\varepsilon>0: \mu\big\{x\in X\;\big|\;|h(x)-s|<\varepsilon\big\}>0\Big\}.
\end{equation}
\end{defi}

In the same way as in the complex case one now proves the following basic properties of the essential range.

\begin{lem}\label{lem_Properties_esssential_range}
Let $X$ be a measure space, $h:X\rightarrow\mathbb{R}^{d+1}$ measurable. Then there is

\begin{enumerate}
\item[i)] $\essran(h)$ is closed, with $\essran(h)\subseteq\overline{\ran}(h)$; \medskip
\item[ii)] $\big\{x\in X\;\big|\;h(x)\notin\essran(h)\big\}$ is a set of measure zero; \medskip
\item[iii)] $\Vert h\Vert_{L^\infty}=\sup(\essran(|h|))$; \medskip
\item[iv)] For every continuous $f:\essran(h)\rightarrow\mathbb{R}_d$, there is $\essran(f\circ h)=\overline{f(\essran(h))}$.
\end{enumerate}
\end{lem}

Before we now state the actual characterization of the $S$-spectrum in terms of the essential range, we need the following preparatory lemma.

\begin{lem}\label{lem_Set_inclusion}
Let $\varepsilon>0$ and $s\in\mathbb{R}^{d+1}$. Then there exists some $\varepsilon_s>0$, such that there holds
\begin{equation}\label{Eq_Set_inclusion}
[s]+U_{\varepsilon_s}(0)\subseteq\big\{q\in\mathbb{R}^{d+1}\;\big|\;|q^2-2s_0q+|s|^2|<\varepsilon^2\big\}\subseteq[s]+U_\varepsilon(0).
\end{equation}
\end{lem}

\begin{proof}
For the first inclusion let $q\in[s]+U_{\varepsilon_s}(0)$, for a not yet specified $\varepsilon_s>0$. I.e., there exists some $t\in[s]$ with $|q-t|<\varepsilon_s$. Clearly it is possible to choose $t$ in the same complex plane as $q$. Hence we can estimate
\begin{align*}
\big|q^2-2s_0q+|s|^2\big|&=\big|q^2-2t_0q+|t|^2\big|=|q-t||q-\overline{t}|\leq|q-t|\big(|q-t|+|t-\overline{t}|\big) \\
&\leq\varepsilon_s\big(\varepsilon_s+2|\Im(t)|\big)=\varepsilon_s\big(\varepsilon_s+2|\Im(s)|\big).
\end{align*}
It is now possible to choose $\varepsilon_s>0$ small enough, such that the right hand side of this inequality is smaller than $\varepsilon^2$. This then proves $|q^2-2s_0q+|s|^2|<\varepsilon^2$. \medskip

For the second inclusion in \eqref{Eq_Set_inclusion}, let $q\in\mathbb{R}^{d+1}$ with $|q^2-2s_0q+|s|^2<\varepsilon^2$. Choose now $t\in[s]$ such that $t$ is in the same complex plane as $q$. Then there is also
\begin{equation*}
|q-t||q-\overline{t}|=|q^2-2s_0q+|s|^2|<\varepsilon^2.
\end{equation*}
From this inequality it the follows immediately that
\begin{align*}
|q-t|^2&\leq|q-t||q-\overline{t}|<\varepsilon^2,\qquad\text{if }|q-t|\leq|q-\overline{t}|, \\
|q-\overline{t}|^2&\leq|q-t||q-\overline{t}|<\varepsilon^2,\qquad\text{if }|q-t|\geq|q-\overline{t}|,
\end{align*}
which means that $q\in U_\varepsilon(t)\cup U_\varepsilon(\overline{t})\subseteq[s]+U_\varepsilon(0)$.
\end{proof}

\begin{thm}\label{thm_Spectrum_of_Mh}
Let $X$ be a $\sigma$-finite measure space, and $h:X\rightarrow\mathbb{R}^{d+1}$ measurable. Then the $S$-spectrum of the multiplication operator $M_h$ is given by
\begin{equation*}
\sigma_S(M_h)=[\essran(h)],
\end{equation*}
where $[\essran(h)]=\big\{s\in\mathbb{R}^{d+1}\;\big|\;[s]\cap\essran(h)\neq\emptyset\big\}$ is the rotation of $\essran(h)$ around the real axis. Moreover, for every $s\in\rho_S(M_h)$ in the $S$-resolvent set, the left $S$-resolvent operator acts as the multiplication operator
\begin{equation}\label{Eq_SL_multiplication}
S_L^{-1}(s,M_h)=M_{S_L^{-1}(s,h(\cdot))}\qquad\text{and}\qquad S_R^{-1}(s,M_h)=M_{S_R^{-1}(s,h(\cdot))}.
\end{equation}
\end{thm}

\begin{proof}
If we use the representation $Q_s[M_h]=M_{h^2-2s_0h+|s|^2}$ from Lemma~\ref{lem_QsMh}, then Theorem~\ref{thm_Bijective} shows that $Q_s[M_h]$ is not bijective if and only if
\begin{equation}\label{Eq_S_spectrum_1}
\mu\big\{x\in X\;\big|\;|h(x)^2-2s_0h(x)+|s|^2|<\varepsilon^2\big\}>0,\qquad\text{for every }\varepsilon>0.
\end{equation}
Here $\mu$ is the measure, corresponding to the space $X$. We now want to show that \eqref{Eq_S_spectrum_1} is equivalent to $s\in[\essran(h)]$. \medskip

For the first implication let $s\in[\essran(h)]$ and $\varepsilon>0$ arbitrary. Then there exists some $t\in\essran(h)$, such that $t\in[s]$. By Lemma~\ref{lem_Set_inclusion} there then exists some $\varepsilon_s>0$, with
\begin{equation*}
U_{\varepsilon_s}(t)\subseteq[s]+U_{\varepsilon_s}(0)\subseteq\big\{q\in\mathbb{R}^{d+1}\;\big|\;|q^2-2s_0q+|s|^2|<\varepsilon^2\big\}.
\end{equation*}
Hence, we also get
\begin{equation*}
\mu\big\{x\in X\;\big|\;|h(x)^2-2s_0h(x)+|s|^2|<\varepsilon^2\big\}\geq\mu\big\{x\in X\;\big|\;|h(x)-t|<\varepsilon_s\big\}>0,
\end{equation*}
where the measure of the set on the right hand side is positive due to $t\in\essran(h)$ and the definition \eqref{Eq_essran} of the essential range. \medskip

For the inverse implication, let \eqref{Eq_S_spectrum_1} be satisfied and assume that $s\notin[\essran(h)]$. This means that $[s]\cap\essran(h)=\emptyset$. Since $[s]$ is compact and $\essran(h)$ is closed, there exists some $\varepsilon>0$, such that
\begin{equation*}
[s]+U_\varepsilon(0)\subseteq\mathbb{R}^{d+1}\setminus\essran(h).
\end{equation*}
By Lemma~\ref{lem_Set_inclusion} there then also is
\begin{equation*}
\big\{q\in\mathbb{R}^{d+1}\;\big|\;|q^2-2s_0q+|s|^2|<\varepsilon^2\big\}\subseteq\mathbb{R}^{d+1}\setminus\essran(h).
\end{equation*}
Since the set in \eqref{Eq_S_spectrum_1} has positive measure by assumption, we also get
\begin{equation*}
\mu\big\{x\in X\;\big|\;h(x)\in\mathbb{R}^{d+1}\setminus\essran(h)\big\}\geq\mu\big\{x\in X\;\big|\;|h(x)^2-2s_0h(x)+|s|^2|<\varepsilon^2\big\}>0.
\end{equation*}
However, since the left hand side has measure zero by Lemma~\ref{lem_Properties_esssential_range}~ii), this is a contradiction. Hence we have proven $s\in[\essran(h)]$. \medskip

In order to verify the explicit form of the left $S$-resolvent operator \eqref{Eq_SL_multiplication}, we first note that by its definition in \eqref{Eq_S_resolvent_operators}, there is
\begin{equation*}
S_L^{-1}(s,M_h)=Q_s[M_h]^{-1}\overline{s}-M_hQ_s[M_h]^{-1}.
\end{equation*}
Using now \eqref{Eq_QsMh}, the form of the inverse multiplication operator in Theorem~\ref{thm_Bijective}, as well as the product formula for the multiplication operator in Lemma~\ref{lem_Properties_multiplication_operator}~iii), we can further simplify
\begin{align*}
S_L^{-1}(s,M_h)&=M_{h^2-2s_0h+|s|^2}^{-1}\overline{s}-M_hM_{h^2-2s_0h+|s|^2}^{-1} \\
&=M_{(h^2-2s_0h+|s|^2)^{-1}}\overline{s}-M_hM_{(h^2-2s_0h+|s|^2)^{-1}} \\
&=M_{(h^2-2s_0h+|s|^2)^{-1}\overline{s}-h(h^2-2s_0h+|s|^2)^{-1}}=M_{S_L^{-1}(s,h(\cdot))}.
\end{align*}
The representation of the right $S$-resolvent operator is analog.
\end{proof}

Now, since Theorem~\ref{thm_Spectrum_of_Mh} tells us how the $S$-spectrum of the multiplication operator looks like, we turn our attention to the $H^\infty$-functional calculus of Definition~\ref{defi_Hinfty_SFC}. To do so, we start with the bisectoriality, according to Definition~\ref{defi_Bisectorial_operators}, of the operator $M_h$.

\begin{thm}
Let $X$ be a $\sigma$-finite measure space, and $h:X\rightarrow\mathbb{R}^{d+1}$ measurable, with
\begin{equation}\label{Eq_Bisectorial_assumption}
\essran(h)\subseteq\overline{D_\omega},
\end{equation}
for some $\omega\in(0,\frac{\pi}{2})$. Then the multiplication operator $M_h$ is bisectorial of angle $\omega$.
\end{thm}

\begin{proof}
Since $\overline{D_\omega}$ is axially symmetric, it follows from Theorem~\ref{thm_Spectrum_of_Mh} together with \eqref{Eq_Bisectorial_assumption}, that
\begin{equation*}
\sigma_S(M_h)=[\essran(h)]\subseteq[\overline{D_\omega}]=\overline{D_\omega}.
\end{equation*}
In order to prove the $S$-resolvent estimate \eqref{Eq_SL_estimate}, let $\varphi\in(\omega,\pi)$ and $s\in\mathbb{R}^{d+1}\setminus(D_\varphi\cup\{0\})$. Using the explicit form of the $S$-resolvent operator \eqref{Eq_SL_multiplication} as well as the operator norm of the multiplication operator in Lemma~ \ref{lem_Properties_multiplication_operator}~ii), we get
\begin{align*}
\Vert S_L^{-1}(s,M_h)\Vert&=\Vert M_{S_L^{-1}(s,h(\cdot))}\Vert\leq 2^{\frac{d}{2}}\Vert S_L^{-1}(s,h(\cdot))\Vert_{L^\infty} \\
&=2^{\frac{d}{2}}\esssup_{x\in X}\big|Q_s(h(x))^{-1}\overline{s}-h(x)Q_s(h(x))^{-1}\big| \\
&\leq 2^{\frac{d}{2}}\esssup_{x\in X}\frac{|s|+|h(x)|}{|h(x)^2-2s_0h(x)+|s|^2|}\leq 2^{\frac{d}{2}}\sup\limits_{q\in\overline{D_\omega}}\frac{|s|+|q|}{|q^2-2s_0q+|s|^2|} \\
&\leq 2^{\frac{d}{2}}\sup\limits_{q\in\overline{D_\omega}}\frac{|s|+|q|}{|q-s|^2}\le\frac{2^{\frac{d}{2}-1}}{|s|\sqrt{1+\cos(\varphi-\omega)}\big(\sqrt{2}-\sqrt{1+\cos(\varphi-\omega)}\big)}. \qedhere
\end{align*}
\end{proof}

As the final part of this application section, we want to derive how the $H^\infty$-functional calculus of Definition~\ref{defi_Hinfty_SFC} acts on the multiplication operator.

\begin{thm}\label{thm_Hinfty_Mh}
Let $X$ be a $\sigma$-finite measure space, and $h:X\rightarrow\mathbb{R}^{d+1}$ measurable, with $\essran(h)\subseteq\overline{D_\omega}$, for some $\omega\in(0,\frac{\pi}{2})$. Then for every $f\in\mathcal{SH}_L^\poly(D_\theta)$ or $f\in\mathcal{SH}_R^\poly(D_\theta)$, $\theta\in(\omega,\frac{\pi}{2})$, there holds
\begin{equation}\label{Eq_Hinfty_Mh}
f(M_h)=M_{f\circ h}.
\end{equation}
\end{thm}

\begin{proof}
Let us start with the $\omega$-functional calculus \eqref{Eq_omega_left} for functions $f\in\mathcal{SH}_L^0(D_\theta)$. Evaluated for $u\in L^2(X)$, and $x\in X$, we have
\begin{align*}
f(M_h)u(x)&=\frac{1}{2\pi}\int_{\partial D_\varphi\cap\mathbb{C}_J}S_L^{-1}(s,M_h)ds_Jf(s)u(x) \\
&=\frac{1}{2\pi}\int_{\partial D_\varphi\cap\mathbb{C}_J}S_L^{-1}(s,h(x))ds_Jf(s)u(x),
\end{align*}
where in the second line we plugged in the action \eqref{Eq_SL_multiplication} of the $S$-resolvent operator. Using now the Cauchy integral formula
\begin{equation*}
f(q)=\frac{1}{2\pi}\int_{\partial D_\varphi\cap\mathbb{C}_J}S_L^{-1}(s,q)ds_Jf(s),\qquad q\in D_\varphi,
\end{equation*}
with $q=h(x)$, gives
\begin{align*}
f(M_h)u(x)=f(h(x))u(x),\qquad u\in L^2(X),\text{ for a.e. }x\in X .
\end{align*}
At least this is true for every $x\in X$ such that $h(x)\in\essran(h)$, which is almost everywhere due to Lemma~\ref{lem_Properties_esssential_range}~iii). Since this is true for every $u\in L^2(X)$ and a.e. $x\in X$, we see that $f(M_h)$ acts as the multiplication operator with the function $f\circ h$, i.e. we have proven \eqref{Eq_Hinfty_Mh} for functions $f\in\mathcal{SH}_L^0(D_\theta)$. The same formula for $f\in\mathcal{SH}_R^0(D_\theta)$ is analog. \medskip

The $H^\infty$-functional calculus \eqref{Eq_Hinfty_Mh} for functions $f\in\mathcal{SH}_L^\poly(D_\theta)$, is defined as
\begin{equation*}
f(M_h)=e(M_h)^{-1}(ef)(M_h).
\end{equation*}
With the already derived identity \eqref{Eq_Hinfty_Mh} for functions in $\mathcal{SH}_L^0(D_\theta)$, in particular for $e$ and $ef$, we can rewrite this definition as
\begin{equation*}
f(M_h)=M_{e\circ h}^{-1}M_{(ef)\circ h}=M_{(e\circ h)^{-1}}M_{(ef)\circ h}=M_{(e\circ h)^{-1}((ef)\circ h)}=M_{f\circ h}.
\end{equation*}
On the other hand, for $f\in\mathcal{SH}_R^\poly(D_\theta)$, the definition \eqref{Eq_Right_Hinfty} of the $H^\infty$-functional calculus for right holommorphic functions tells us that
\begin{equation*}
f(M_h)=\overline{(fe)(M_h)e(M_h)^{-1}}=\overline{M_{(ef)\circ h)(e\circ h)^{-1}}}=M_{((ef)\circ h)(e\circ h)^{-1}}=M_{f\circ h},
\end{equation*}
where, additionally to the left case, we used that the multiplication operator is closed.
\end{proof}

\begin{cor}
Let $X$ be a $\sigma$-finite measure space, and $h:X\rightarrow\mathbb{R}^{d+1}$ measurable, with $\essran(h)\subseteq\overline{D_\omega}$, for some $\omega\in(0,\frac{\pi}{2})$. Then for every $f\in\mathcal{SH}_L^\poly(D_\theta)$ or $f\in\mathcal{SH}_R^\poly(D_\theta)$, $\theta\in(\omega,\frac{\pi}{2})$, there is
\begin{equation*}
f(M_h)^*=M_{\overline{f}\circ h}.
\end{equation*}
\end{cor}

\begin{proof}
Consecutively applying the results of Theorem~\ref{thm_Hinfty_Tstar}, Theorem~\ref{thm_Adjoint_multiplication_operator}, and Theorem~\ref{thm_Hinfty_Mh}, as well as the definition of $f^\#$ in \eqref{Eq_fsharp}, gives
\begin{equation*}
f(M_h)^*=f^\#(M_h^*)=f^\#(M_{\overline{h}})=M_{f^\#\circ\overline{h}}=M_{\overline{f}\circ h}. \qedhere
\end{equation*}
\end{proof}

\section{Concluding remarks on applications and research directions}\label{SEC_CONCLUD_RMK}

The spectral theory based on the $S$-spectrum was inspired by quaternionic quantum mechanics and began its development in 2006. A comprehensive introduction can be found in \cite{CGK}, with further explorations done in \cite{ACS2016,AlpayColSab2020,ColomboSabadiniStruppa2011}. Applications on fractional powers of operators are investigated in \cite{CGdiffusion2018,CG18,FJBOOK,JONAMEM,JONADIRECT} and some results from classical interpolation theory have been recently extended into this setting \cite{COLSCH}. The spectral theory on the $S$-spectrum extends complex spectral theory to noncommuting operators, broadening its applicability to various research fields. Without claiming completeness we mention: \medskip

\textit{Quaternionic formulation of quantum mechanics.} The interest in spectral theory for quaternionic operators is motivated by the 1936 paper \cite{BF} on the logic of quantum mechanics by G. Birkhoff, J. von Neumann. There, the authors showed that Schrödinger equation can be written basically in the complex or in the quaternionic setting, see also the book of Adler \cite{adler}. \medskip

\textit{Vector analysis.} It is the natural spectral theory for operaotrs appearing in vector analysis such as the gradient operator with nonconstant coefficients in $n$ dimensions, represents various physical laws, such as Fourier's law for heat propagation and Fick's law for mass transfer diffusion. This can be expressed as
\begin{equation*}
T=\sum\nolimits_{i=1}^ne_ia_i(x)\partial_{x_i},\qquad x\in\Omega,
\end{equation*}
and is associated with various boundary conditions, as discussed in \cite{CMS24}. \medskip

\textit{Differential geometry.} This theory applies to Dirac operators on manifold. For more details study the reader can refer to the book \cite{DiracHarm}. As a special case the paper \cite{DIRACHYPSPHE} considers the Dirac operator in hyperbolic and spherical spaces, that takes the explicit forms
\begin{equation*}
\mathcal{D}_H=x_n\sum\nolimits_{i=1}^ne_i\partial_{x_i}-\frac{n-1}{2}e_n\qquad\text{and}\qquad\mathcal{D}_S=(1+|x|^2)\sum\nolimits_{i=1}^ne_i\partial_{x_i}-nx,
\end{equation*}
where for $\mathcal{D}_H$ we pick $(x_1,\dots,x_{n-1},x_n)\in\mathbb{R}^{n-1}\times\mathbb{R}^+$ and $(x_1,...,x_n)\in\mathbb{R}^n$ for $\mathcal{D}_S$. \medskip

\textit{Hypercomplex analysis.} The well-known Dirac operator and its conjugate
\begin{equation*}
D=\partial_{x_0}+\sum\nolimits_{i=1}^ne_i\partial_{x_i}\qquad\text{and}\qquad\overline{D}=\partial_{x_0}-\sum\nolimits_{i=1}^ne_i\partial_{x_i},
\end{equation*}
are widely investigated in \cite{DSS} as well as in \cite{DiracHarm}. For integers $\alpha,\beta,m$, the operators of the Dirac fine structure on the $S$-spectrum are defined as
\begin{equation*}
T_{\alpha,m}=D^\alpha(D\overline{D})^m\qquad\text{and}\qquad\widetilde{T}_{\beta,m}=\overline{D}^\beta(D\overline{D})^m.
\end{equation*}
The fine structures on the $S$-spectrum constitute a set of function spaces and their related functional calculi. These structures have been introduced and studied in recent works \cite{polypolyFS,CDPS1,Fivedim,Polyf1,Polyf2}, while their respective $H^\infty$-versions are investigated in \cite{MILANJPETER,MPS23}. \medskip

The $H^\infty$-functional calculi also have a broader context in the recently introduced fine structures on the $S$-spectrum. These are function spaces of nonholomorphic functions derived from the Fueter-Sce extension theorem, which connects slice hyperholomorphic and axially monogenic functions, the two main notions of holomorphicity in the Clifford setting. The connection is established through powers of the Laplace operator in a higher dimension, see \cite{Fueter,TaoQian1,Sce} and also the translation \cite{ColSabStrupSce}. \medskip

It is important to highlight that in the hypercomplex setting there exists another spectral theory based on the monogenic spectrum, which was initiated by Jefferies, McIntosh and J. Picton-Warlow in \cite{JM}. This theory is founded on monogenic functions, which are functions in the kernel of the Dirac operator, and their associated Cauchy formula (see \cite{DSS} for a comprehensive treatment). The monogenic spectral theory, based on monogenic functions \cite{DSS}, has been extensively developed and is well described in the seminal books \cite{JBOOK} and \cite{TAOBOOK}. These texts provide an in-depth exploration of the theory, with particular emphasis on the $H^\infty$-functional calculus in the monogenic setting.

\section*{Declarations and statements}

\textbf{Data availability}. The research in this paper does not imply use of data. \medskip

\textbf{Conflict of interest}. The authors declare that there is no conflict of interest. \medskip

\textbf{Acknowledgements:} Fabrizio Colombo is supported by MUR grant Dipartimento di Eccellenza 2023-2027. Peter Schlosser was funded by the Austrian Science Fund (FWF) under Grant No. J 4685-N and by the European Union--NextGenerationEU.

\end{document}